\documentclass[11pt,a4paper,reqno]{amsart}
\pdfoutput=1
\usepackage{amssymb}
\usepackage{amsmath}
\usepackage[alphabetic]{amsrefs}
\usepackage{tikz-cd}
\usepackage{hyperref}
\usepackage{mathtools}
\usepackage{microtype}
\usepackage{lmodern}
\usepackage[headings]{fullpage}

\theoremstyle{plain}
\newtheorem{theorem}[subsection]{Theorem}
\newtheorem{proposition}[subsection]{Proposition}
\newtheorem{lemma}[subsection]{Lemma}
\newtheorem{corollary}[subsection]{Corollary}

\theoremstyle{definition}
\newtheorem{definition}[subsection]{Definition}

\theoremstyle{remark}
\newtheorem{remark}[subsection]{Remark}

\numberwithin{equation}{section}

\newcommand{\tenscorep}{\mathbin{\begin{tikzpicture}[baseline,x=.75ex,y=.75ex] \draw (-0.8,1.15)--(0.8,1.15);\draw(0,-0.25)--(0,1.15); \draw (0,0.75) circle [radius = 1];\end{tikzpicture}}}

\title{Braided Quantum Symmetries of Graph \texorpdfstring{$\textup{C}^*$}{}-algebras}

\author[Bhattacharjee]{Suvrajit Bhattacharjee}
\address{Matematisk institutt, Universitetet i Oslo, P.O. Box 1053, Blindern, 0316 Oslo, Norway}
\email{suvrajib@math.uio.no}
\author[Joardar]{Soumalya Joardar}
\address{Department of Mathematics and Statistics, Indian Institute of Science Education and Research Kolkata, Mohanpur - 741246, West
Bengal, India}
\email{soumalya.j@gmail.com}
\author[Roy]{Sutanu Roy}
\address{School of Mathematical Sciences, National Institute of Science Education and Research Bhubaneswar, Jatni, 752050, India}
\address{Homi Bhabha National Institute, Training School Complex, Anushaktinagar, Mumbai, 400094, India}
\email{sutanu@niser.ac.in}

\begin{document}

\begin{abstract}
We prove the existence of a universal braided compact quantum group acting on a graph $\textup{C}^*$\nobreakdash-algebra in the category of $\mathbb{T}$\nobreakdash-$\textup{C}^*$\nobreakdash-algebras with a twisted monoidal structure, in the spirit of the seminal work of S. Wang. To achieve this, we construct a braided analogue of the free unitary quantum group and study its bosonization. As a concrete example, we compute this universal braided compact quantum group for the Cuntz algebra.     
\end{abstract}



\maketitle

\section{Introduction}

Since the work of Manin in \cite{manin}, the concept of quantum symmetries of a space, in both classical and noncommutative sense, has been (and still being) thoroughly investigated and clarified in a number of works. It is shown in \cite{manin} that $\textup{SL}_q(2,\mathbb{C})$, one of the very first examples of a quantum group (\cite{drinfeld}) in Drinfeld's sense, i.e., a noncommutative and noncocommutative Hopf algebra, appears as a quantum symmetry group of a noncommutative space. Woronowicz's theory of compact matrix pseudogroups (\cite{woropseudo}), describing a compact topological group (of matrices) in a purely $\textup{C}^*$\nobreakdash-algebraic language, marked a new beginning which is also better suited to Connes' approach to noncommutative geometry (\cite{connes}). 

A landmark example (\cite{worosu}) constructed by Woronowicz is the compact quantum group $\textup{SU}_q(2)$ for each $0 < q \leq 1$. Thus for each $q \in (0,1]$, one has a unital $\textup{C}^*$\nobreakdash-algebra $\textup{C}(\textup{SU}_q(2))$ and a unital $*$-homomorphism $\Delta_{\textup{SU}_q(2)}$ from $\textup{C}(\textup{SU}_q(2))$ to $\textup{C}(\textup{SU}_q(2)) \otimes \textup{C}(\textup{SU}_q(2))$ satisfying some natural conditions. Setting $q=1$, one recovers the algebra of continuous functions on the compact group $\textup{SU}(2)$ and the morphism induced by the group-multiplication.  The discovery of $\textup{SU}_q(2)$ together with the dream of making contact with Connes' enterprise, resulted, following Wang's pioneering work on quantum symmetries of finite spaces (\cite{wang}), in several thematically entwined constructions and insights. We mention, as a necessarily incomplete sampling,

\begin{enumerate}
    \item Banica, Bichon and collaborators on quantum symmetries of discrete structures, see \cites{banicagraphs,banicametric,bichongraph};
    \item Goswami, Bhowmick and collaborators on quantum isometries of spectral triples, see \cites{banicagoswami,goswamiadv,GJ2018,bg2019,BG2009a,bbg};
    \item Banica, Skalski and collaborators on quantum symmetries of $\textup{C}^*$\nobreakdash-algebras equipped with orthogonal filtrations, see \cites{bs2013,bmrs2019,tdec};
    \item and more recently, Goswami and collaborators on quantum symmetries of subfactors, see \cites{bcg2022,HG2021}. 
\end{enumerate}

The deformation parameter $q$ in the definition of $\textup{C}(\textup{SU}_q(2))$ can be relaxed so as to require it to be a nonzero real number $q \in \mathbb{R}^{\times}$. Letting $q$ to be any nonzero complex number $q \in \mathbb{C} \setminus \mathbb{R}$, other than the reals, unlocks a plethora of interesting phenomena; for starter, the comultiplication $\Delta_{\textup{SU}_q(2)}$ does not take values in the minimal tensor product $\textup{C}(\textup{SU}_q(2)) \otimes \textup{C}(\textup{SU}_q(2))$ anymore. What the definition of a compact quantum group misses in this case is a fine structure of the underlying $\textup{C}^*$\nobreakdash-algebra $\textup{C}(\textup{SU}_q(2))$; it is the hidden $\mathbb{T}$\nobreakdash-structure on $\textup{C}(\textup{SU}_q(2))$, plus a twisting by the nontrivial bicharacter on $\mathbb{Z}$ governed by the unit complex number $\zeta=q/\overline{q}$, which is $1$, of course, when $q$ is real. It turns out that the receptacle of the morphism $\Delta_{\textup{SU}_q(2)}$ in this case is the braided tensor product for the $\mathbb{T}$\nobreakdash-structure mentioned just above, in the sense of \cites{mrw2014,mrw2016,R2022}, which, as expected, becomes the minimal tensor product when $q$ is real.

Braided tensor product of two $\textup{C}^*$\nobreakdash-algebras were systematically studied by Meyer, the third author and Woronowicz in \cites{mrw2014,mrw2016} (though there are predecessors, for instance \cite{vaes}) to accommodate several existing tensor products and crossed product-like constructions. For a quasitriangular quantum group $G$, which is a suitable generalization of the algebraic case, this braided tensor product is then used to introduce a monoidal structure $\boxtimes$ on the category of $G$\nobreakdash-$\textup{C}^*$\nobreakdash-algebras. With this in hand, a braided compact quantum group is then defined (\cite{mrw2016}) exactly as an ordinary compact quantum group, the only difference being that the minimal tensor product $\otimes$ is replaced by the braided tensor product $\boxtimes$. $\textup{SU}_q(2)$ for a complex deformation parameter $q$ is a braided compact quantum group (\cite{kmrw2016}) in this sense with the quasitriangular quantum  group $G$ being the circle group $\mathbb{T}$ and so is the braided free orthogonal quantum group, constructed in \cite{MR2021}. 

The next natural step, that of viewing these braided compact quantum groups as symmetry objects of suitable spaces, is explored by the third author in \cite{R2021} for finite spaces, obtaining further examples of braided compact quantum groups (see also \cite{soltan}). Our aim in this paper is to continue along this line. The class of (noncommutative) spaces considered in this paper is that of graph $\textup{C}^*$\nobreakdash-algebras. The relatively recent study of quantum symmetries of this well-studied class of $\textup{C}^*$\nobreakdash-algebras were taken up by the second author and Mandal in \cites{JM2018,JM2021}. A graph $\textup{C}^*$\nobreakdash-algebra carries a natural generalized gauge action and a canonical equivariant state which is $\mathrm{KMS}$ (see, for instance \cite{br}), under some mild restrictions on the underlying graph. These two structures are crucial to our existence result, which is otherwise an extension of the one obtained in \cite{JM2021}. However, there are several technical and conceptual difficulties that are to be overcome. We defer introducing the relevant notations and definitions, but state the main result now.

\begin{theorem}
Let $E=(E^0,E^1,r,s)$ be a finite, directed graph without sinks such that the $\mathrm{KMS}$ state exists. Then there is a universal braided compact quantum group $\mathrm{Qaut}(\textup{C}^*(E))$ acting linearly, faithfully on $\textup{C}^*(E)$ and preserving the $\mathrm{KMS}$ state.
\end{theorem}

In the ordinary compact quantum group case, the path followed in \cite{JM2021} to reach this result uses the existence of the free unitary quantum group $\textup{U}^+(F)$, $(F \in \textup{GL}(n,\mathbb{C}))$, a braided version of which is missing hitherto. We recall that the free unitary quantum group is defined by only demanding the fundamental representation to be unitary and its conjugate to be equivalent to a unitary representation. The first obstacle in defining a braided analogue of $\textup{U}^+(F)$ is the definition of the conjugate representation. We achieved this in a previous version of this article. However, soon afterwards, a general framework was provided in \cite{cocycle-twist} for constructing this braided compact quantum group along with several others (for instance \cite{anyonic}). Finally, we explicitly compute $\mathrm{Qaut}(\textup{C}^*(E))$ when $\textup{C}^*(E)$ is the Cuntz algebra $\mathcal{O}_n$; in this case $\mathrm{Qaut}(\textup{C}^*(E))$ turns out to be $\textup{U}_{\zeta}^+(I_n)$.

\begin{theorem}
The braided quantum symmetry group for the Cuntz algebra $\mathcal{O}_n$ is the free unitary quantum group $\textup{U}_{\zeta}^+(I_n)$, $I_n$ being the $n \times n$ identity matrix.    
\end{theorem}

The demonstration of the above theorem is naturally divided into two parts; the first part shows that $\textup{U}_{\zeta}^+(I_n)$ is indeed a candidate to be the braided quantum symmetry group (in technical terms it is an object of a suitable category) and the second part shows its universality. We remark that the proof of the first part (Proposition \ref{objectproof}) is an improvement over the proof as given  in \cites{JM2018,JM2021} in the unbraided case. 

Let us now briefly describe the organization of this paper. In Section \ref{freeunitary}, we collect together all the results pertaining to the braided free unitary quantum group. In Section \ref{graph}, we prove the main existence result on the braided quantum symmetry group of the graph $\textup{C}^*$\nobreakdash-algebra $\textup{C}^*(E)$ (Theorem \ref{graphsymm}) and in Section \ref{sec:cuntz-algebra}, we compute it for the Cuntz algebra $\mathcal{O}_n$ (Theorem \ref{cuntzsymm}). 

\section*{Acknowledgments} The first author was supported by the NFR project 300837 ``Quantum Symmetry'' and the Charles University PRIMUS grant \textit{Spectral Noncommutative Geometry of Quantum Flag Manifolds} PRIMUS/21/SCI/026. He is grateful to Indian Statistical Institute, Kolkata and Prof. Debashish Goswami for offering him a visiting scientist position, where this work was started. He is also grateful to SERB, Government of India (PDF/2021/003544) and Indian Institute of Science Education and Research, Kolkata for hosting him, where this work was written up. The second author was partially supported by SERB MATRICS grant MTR/2022/000515. He was also partially supported by INSPIRE faculty award DST/INSPIRE/04/2016/002469 given by the D.S.T., Government of India, when this work was started. The third author was partially supported by DST/INSPIRE/04/2016/000215, given by the D.S.T., Government of India.

\section*{Notations} For two $\textup{C}^*$\nobreakdash-algebras $A$ and $B$, $A \otimes B$ denotes the minimal tensor product of $\textup{C}^*$\nobreakdash-algebras. For a $\textup{C}^*$\nobreakdash-algebra $A$ and two closed subspaces $X,Y \subset A$, $XY$ denotes the norm-closed linear span of the set of products $xy$, $x \in X$ and $y \in Y$. For an object $X$ in some category, $\mathrm{id}_X$ denotes the identity morphism of $X$. For a unital $\textup{C}^*$\nobreakdash-algebra $A$, $1_A$ denotes the unit element in $A$, and $\mathcal{M}(A)$ denotes the multiplier algebra of $A$.

\section{The braided free unitary quantum group}\label{freeunitary} In this section, we gather some preliminaries regarding braided compact quantum groups over $\mathbb{T}$ (see \cites{mrw2014,mrw2016,kmrw2016}), presented in a way that naturally leads us to the definition of the braided free unitary quantum group. We omit the proofs as they are contained in a much more general framework in \cite{cocycle-twist}. Let $\mathcal{C}^*$ be the category of $\textup{C}^*$\nobreakdash-algebras. For $A$ and $B$ in $\mathrm{Obj}(\mathcal{C}^*)$, we write the set of morphisms as $\mathrm{Mor}(A,B)$ which consists of $*$-homomorphisms $\pi : A \rightarrow \mathcal{M}(B)$ such that $\pi(A)B=B$, where $\mathcal{M}(B)$ is the multiplier algebra of $B$. Thus for unital $A$ and $B$, $\mathrm{Mor}(A,B)$ consists of unital $*$-homomorphisms from $A$ to $B$. We recall that a compact quantum group $G$ consists of a pair $G=(\textup{C}(G),\Delta_{G})$ where $\textup{C}(G)$ is a unital $\textup{C}^*$\nobreakdash-algebra and $\Delta_{G} : \textup{C}(G) \rightarrow \textup{C}(G) \otimes \textup{C}(G)$ is a coassociative morphism satisfying a cancellation property. We also recall that the comultiplication $\Delta_{\mathbb{T}} : \textup{C}(\mathbb{T}) \rightarrow \textup{C}(\mathbb{T}) \otimes \textup{C}(\mathbb{T})$ of the compact quantum group $\textup{C}(\mathbb{T})$ sends $z$ to $z \otimes z$.

\begin{definition}\label{circleaction}
We define the category $\mathcal{C}^*_{\mathbb{T}}$ of $\mathbb{T}$\nobreakdash-$\textup{C}^*$\nobreakdash-algebras and $\mathbb{T}$\nobreakdash-equivariant morphisms as follows. An object of $\mathcal{C}^*_{\mathbb{T}}$ is a pair $(X,\rho^X)$, where $X$ is a unital $\textup{C}^*$\nobreakdash-algebra and $\rho^X \in \mathrm{Mor}(X, X \otimes \textup{C}(\mathbb{T}))$ such that
\begin{enumerate}
\item $(\rho^X \otimes \mathrm{id}_{\textup{C}(\mathbb{T})})\circ \rho^X=(\mathrm{id_X} \otimes \Delta_{\mathbb{T}})\circ \rho^X$;
\item $\rho^X(X)(1_X \otimes \textup{C}(\mathbb{T}))=X \otimes \textup{C}(\mathbb{T})$.
\end{enumerate}
Let $(X,\rho^X)$ and $(Y,\rho^Y)$ be two $\mathbb{T}$\nobreakdash-$\textup{C}^*$\nobreakdash-algebras. A morphism $\phi : (X,\rho^X) \rightarrow (Y,\rho^Y)$ in $\mathcal{C}^*_\mathbb{T}$ (or equivalently, a $\mathbb{T}$\nobreakdash-equivariant morphism) is by definition a $\phi \in \mathrm{Mor}(X,Y)$ such that $\rho^Y \circ \phi=(\phi \otimes \mathrm{id}_{\textup{C}(\mathbb{T})}) \circ \rho^X$. We write $\mathrm{Mor}^{\mathbb{T}}(X,Y)$ for the set of morphisms between $(X,\rho^X)$ and $(Y,\rho^Y)$ in $\mathcal{C}^*_{\mathbb{T}}$.
\end{definition}

A $\mathbb{T}$\nobreakdash-$\textup{C}^*$\nobreakdash-algebra $X$ (with $\rho^X$ understood) comes with an associated $\mathbb{Z}$-grading defined as follows. We call an element $x \in X$ homogeneous of degree $n \in \mathbb{Z}$ if $\rho^X(x)=x \otimes z^n$ and write $\mathrm{deg}(x)=n$. For each $n \in \mathbb{Z}$, we let $X(n)$ denote the set consisting of homogeneous elements of degree $n$: $X(n)=\{x \in X \mid \mathrm{deg}(x)=n\}$. The collection $\{X(n)\}_{n \in \mathbb{Z}}$ enjoys the following: 
\begin{enumerate}
\item for each $n \in \mathbb{Z}$, $X(n)$ is a closed subspace of $X$;
\item for $m,n \in \mathbb{Z}$, $X(m)X(n) \subseteq X(m+n)$;
\item for each $n \in \mathbb{Z}$, $X(n)^*=X(-n)$;
\item the algebraic direct sum $\bigoplus_{n \in \mathbb{Z}}X(n)$ is norm-dense $X$.
\end{enumerate}

Fixing $\zeta \in \mathbb{T}$ amounts to fixing an $\mathrm{R}$\nobreakdash-matrix on~\(\mathbb{Z}\), in the sense of \cite{mrw2016}*{Defintion 2.1}. Consequently, $\mathcal{C}^*_{\mathbb{T}}$ gets endowed with a monoidal structure or the braided tensor product~$\boxtimes_{\zeta}$ depending on a parameter $\zeta \in \mathbb{T}$. Let $(X,\rho^X)$ and $(Y,\rho^Y)$ be two objects of $\mathcal{C}^*_{\mathbb{T}}$. Then the braided tensor~\(X\boxtimes_{\zeta}Y\) along with a diagonal action~\(\rho^{X\boxtimes_{\zeta}Y}\) of~\(\mathbb{T}\) becomes an object of $\mathcal{C}^*_{\mathbb{T}}$~\cite{mrw2016}*{Proposition 4.2}. Consider the canonical embeddings~\(j_{1}\in\mathrm{Mor}^{\mathbb{T}}(X,X\boxtimes_{\zeta}Y)\) and \(j_{2}\in\mathrm{Mor}^{\mathbb{T}}(Y,X\boxtimes_{\zeta}Y)\). We sometimes write $x \boxtimes_{\zeta}1_Y$ for $j_1(x)$ (and similarly, $1_X \boxtimes y=j_2(y)$). The pointwise noncommutativity of \(j_{1}\) and \(j_{2}\) for homogeneous $x$ and $y$ is given by
\begin{align*}
(x \boxtimes_{\zeta} 1_Y)(1_X \boxtimes_{\zeta} y)=j_1(x)j_2(y)={}&\zeta^{-\mathrm{deg}(x)\mathrm{deg}(y)}j_2(y)j_1(x)\\
={}&\zeta^{-\mathrm{deg}(x)\mathrm{deg}(y)}(1_X \boxtimes_{\zeta} y)(x \boxtimes_{\zeta} 1_Y).
\end{align*}At the algebraic level, if $x \in X(k)$ and $y \in Y(l)$ then $j_1(x)j_2(y) \in (X \boxtimes_{\zeta} Y)(k+l)$. Throughout the paper, $X \boxtimes_{\zeta} Y$ is equipped with this $\mathbb{T}$\nobreakdash-action. Suppose we are given two $\mathbb{T}$\nobreakdash-equivariant morphisms $\pi_1 \in \mathrm{Mor}^{\mathbb{T}}(X_1, Y_1)$ and $\pi_2 \in \mathrm{Mor}^{\mathbb{T}}(X_2, Y_2)$. There is a unique $\mathbb{T}$\nobreakdash-equivariant morphism $\pi_1 \boxtimes_{\zeta} \pi_2 \in \mathrm{Mor}^{\mathbb{T}}(X_1 \boxtimes_{\zeta} X_2, Y_1 \boxtimes_{\zeta} Y_2)$ such that $(\pi_1 \boxtimes_{\zeta} \pi_2)(j_1(x_1)j_2(x_2))=j_1(\pi_1(x_1))j_2(\pi_2(x_2))$, for $x_1 \in X_1$ and $x_2 \in X_2$. Now we can define a braided compact quantum group.

\begin{definition}\label{bcqg}\cite{mrw2016}
A braided compact quantum group (over $\mathbb{T}$) $G$ is a triple $G=(\textup{C}(G),\rho^{\textup{C}(G)}, \Delta_G)$, where $\textup{C}(G)$ is a unital $\textup{C}^*$\nobreakdash-algebra, $\rho^{\textup{C}(G)}$ is a $\mathbb{T}$\nobreakdash-action on $\textup{C}(G)$ so that $(\textup{C}(G),\rho^{\textup{C}(G)})$ is an object of $\mathcal{C}^*_{\mathbb{T}}$, $\Delta_{G}$ is a $\mathbb{T}$\nobreakdash-equivariant morphism $\Delta_G \in \mathrm{Mor}^{\mathbb{T}}(\textup{C}(G), \textup{C}(G) \boxtimes_{\zeta} \textup{C}(G))$ such that
\begin{enumerate}
    \item $(\Delta_G \boxtimes_{\zeta} \mathrm{id}_{\textup{C}(G)}) \circ \Delta_G=(\mathrm{id}_{\textup{C}(G)} \boxtimes_{\zeta} \Delta_G)\circ \Delta_G$ (coassociativity);
    \item $\Delta_G(\textup{C}(G))j_2(\textup{C}(G))=\Delta_G(\textup{C}(G))j_1(\textup{C}(G))=\textup{C}(G) \boxtimes_{\zeta} \textup{C}(G)$ (bisimplifiability).
\end{enumerate}
\end{definition}

Before defining a representation, we need one more notation. Let $(A, \rho^A) \in \mathrm{Obj}(\mathcal{C}^*_{\mathbb{T}})$ be a $\mathbb{T}$\nobreakdash-$\textup{C}^*$\nobreakdash-algebra. For $a \in A$ and $z \in \mathbb{T}$, we write the value of the map $\rho^A(a) \in A \otimes \textup{C}(\mathbb{T}) \cong \textup{C}(\mathbb{T},A)$ at $z$ as $\rho^A_z(a)$, i.e., $\rho^A(a)(z)=\rho^A_z(a) \in A$. Let $\mathcal{H}$ be a Hilbert space equipped with a (strongly continuous) unitary representation of $\mathbb{T}$, $U : \mathbb{T} \rightarrow \mathcal{U}(\mathcal{H})$. We define 
\[
\rho^{\mathcal{K}(\mathcal{H})} : \mathcal{K}(\mathcal{H}) \rightarrow \mathcal{K}(\mathcal{H}) \otimes \textup{C}(\mathbb{T}), \quad \rho_z^{\mathcal{K}(\mathcal{H})}(x)=U(z)xU(z)^*, \quad z \in \mathbb{T}, x \in \mathcal{K}(\mathcal{H}),
\]making $(\mathcal{K}(\mathcal{H}),\rho^{\mathcal{K}(\mathcal{H})})$ a $\mathbb{T}$\nobreakdash-$\textup{C}^*$\nobreakdash-algebra. It is shown in \cite{mrw2014}*{Corollary 5.16} that for a $\mathbb{T}$\nobreakdash-$\textup{C}^*$\nobreakdash-algebra $(A,\rho^A)$, the braided tensor product $\mathcal{K}(\mathcal{H}) \boxtimes_{\zeta} A$ can be identified with the minimal tensor product $\mathcal{K}(\mathcal{H}) \otimes A$ and we identify them henceforth.

\begin{definition}\label{rep}\cite{MR2019}
Let $G=(\textup{C}(G),\rho^{\textup{C}(G)},\Delta_G)$ be a braided compact quantum group. A (unitary) representation of $G$ is a triple $(\mathcal{H},U,u)$, where $\mathcal{H}$ is a Hilbert space, $U$ is a strongly continuous unitary representation of $\mathbb{T}$ on $\mathcal{H}$ and $u \in \mathcal{M}(\mathcal{K}(\mathcal{H}) \otimes \textup{C}(G))$ is a (unitary) element such that
\begin{enumerate}
    \item for each $z \in \mathbb{T}$, $(\rho^{\mathcal{K}(\mathcal{H})}_z \otimes  \rho^{\textup{C}(G)}_z)(u)=u$, i.e., $u$ is $\mathbb{T}$\nobreakdash-invariant;
    \item $(\mathrm{id_{\mathcal{H}}} \otimes \Delta_G)(u)=(\mathrm{id}_{\mathcal{H}} \otimes j_1)(u)(\mathrm{id}_{\mathcal{H}} \otimes j_2)(u)$.
\end{enumerate}
\end{definition}

Focussing on the case when $\mathcal{H}$ is finite dimensional, one has the following proposition. 

\begin{proposition}\label{equivrep}
Let $G=(\textup{C}(G),\rho^{\textup{C}(G)},\Delta_G)$ be a braided compact quantum group and $(\mathcal{H},U)$ be a pair consisting of a finite dimensional Hilbert space $\mathcal{H}$ and a strongly continuous unitary representation $U$ of $\mathbb{T}$ on $\mathcal{H}$. Then a matrix $u=(u_{ij})_{1 \leq i,j \leq n} \in \mathcal{M}(\mathcal{K}(\mathcal{H}) \otimes \textup{C}(G))=\mathcal{L}(\mathcal{H}) \otimes \textup{C}(G)$ is a unitary representation of $G$ if and only if 
\begin{enumerate}
    \item $\sum_{k=1}^n u^*_{ki}u_{kj}=\delta_{ij}=\sum_{k=1}^nu_{ik}u^*_{jk},$ for $1 \leq i,j \leq n$;
    \item $\rho^{\textup{C}(G)}_z(u_{ij})=z^{d_j-d_i}u_{ij}$, for $1 \leq i,j \leq n$;
    \item $\Delta_G(u_{ij})=\sum_{k=1}^nj_1(u_{ik})j_2(u_{kj})$, for $1 \leq i,j \leq n$.
\end{enumerate}
\end{proposition}

In the ordinary compact quantum group case, $u=(u_{ij})_{1 \leq i,j \leq n}$ being a (unitary) representation automatically implies that the matrix conjugate $\overline{u}=(u^*_{ij})_{1 \leq i,j \leq n}$ is also (equivalent to) a (unitary) representation. However, in the braided case, this is not so.

\begin{definition}\label{conj}
Let $G=(\textup{C}(G),\rho^{\textup{C}(G)},\Delta_G)$ be a braided compact quantum group and let $(\mathcal{H},U,u)$ be a finite dimensional representation of $G$. The conjugate representation corresponding to $u$, denoted as $(\overline{\mathcal{H}},\overline{U},\overline{u}_{\zeta})$ (as opposed to the standard $\overline{u}$, to emphasize the role of $\zeta$) is defined to be the matrix $\overline{u}_{\zeta}=(\zeta^{d_i(d_j-d_i)}u^*_{ij})_{1 \leq i,j \leq n}$.
\end{definition}

Now we proceed to define the braided free unitary quantum group.

\begin{definition}\label{braidedunitary}
Let $F \in \textup{GL}(n,\mathbb{C})$ be a diagonal matrix. We define $\textup{C}(\textup{U}_{\zeta}^{+}(F))$ to be the universal unital $\textup{C}^*$\nobreakdash-algebra with generators $u_{ij}$ for $1 \leq i,j \leq n$ subject to the relations that make $u$ and $F\overline{u}_{\zeta}F^{-1}$ unitaries, where $u=(u_{ij})_{1 \leq i,j \leq n}$ and $\overline{u}_{\zeta}$ as in Definition \ref{conj}.
\end{definition}

One then has the following two propositions.

\begin{proposition}
There is a unique unital $*$-homomorphism \[\rho^{\textup{C}(\textup{U}_{\zeta}^{+}(F))} : \textup{C}(\textup{U}_{\zeta}^{+}(F)) \rightarrow \textup{C}(\mathbb{T}, \textup{C}(\textup{U}_{\zeta}^{+}(F)))\] such that $\rho_z^{\textup{C}(\textup{U}_{\zeta}^{+}(F))}(u_{ij})=z^{d_j-d_i}u_{ij}$ for $1 \leq i,j \leq n$ and $z \in \mathbb{T}$, satisfying the two conditions in Definition \textup{\ref{circleaction}}, making $(\textup{C}(\textup{U}_{\zeta}^{+}(F)),\rho^{\textup{C}(\textup{U}_{\zeta}^{+}(F))})$ a $\mathbb{T}$\nobreakdash-$\textup{C}^*$\nobreakdash-algebra.
\end{proposition}

\begin{proposition}\label{coproduct}
There is a unique unital $*$-homomorphism \[\Delta_{\textup{U}_{\zeta}^{+}(F)} : \textup{C}(\textup{U}_{\zeta}^{+}(F)) \rightarrow \textup{C}(\textup{U}_{\zeta}^{+}(F)) \boxtimes_{\zeta} \textup{C}(\textup{U}_{\zeta}^{+}(F))\] such that $\Delta_{\textup{U}_{\zeta}^{+}(F)}(u_{ij})=\sum_{k=1}^nj_1(u_{ik})j_2(u_{kj})$ for $1 \leq i,j \leq n$. Furthermore, $\Delta_{\textup{U}_{\zeta}^{+}(F)}$ is $\mathbb{T}$\nobreakdash-equivariant, coassociative and bisimplifiable \textup{(}see Definition \textup{\ref{bcqg})}.
\end{proposition}

Summarizing, we make the following definition.

\begin{definition}\label{thequantum}
We define the braided free unitary quantum group, denoted $\textup{U}_{\zeta}^{+}(F)$, to be the braided compact quantum group $(\textup{C}(\textup{U}_{\zeta}^{+}(F)),\rho^{\textup{C}(\textup{U}_{\zeta}^{+}(F))},\Delta_{\textup{U}_{\zeta}^{+}(F)})$, constructed above. 
\end{definition}

We now come to the bosonization construction which gives an equivalence between the category of braided compact quantum groups and the category of ordinary compact quantum groups together with an idempotent quantum group homomorphism. We shall describe explicitly the bosonization of the braided free unitary quantum group constructed above, to be used in the sequel; however, we omit the proofs, as they are contained in \cites{mrw2016,MR2019,cocycle-twist}. 
%
%
%
%

\begin{theorem}\label{explicitboso}
Let $\textup{U}_{\zeta}^+(F) \rtimes \mathbb{T}=(\textup{C}(\textup{U}_{\zeta}^+(F) \rtimes \mathbb{T}),\Delta_{\textup{U}_{\zeta}^+(F) \rtimes \mathbb{T}})$ be the bosonization of the braided free unitary quantum group $\textup{U}_{\zeta}^+(F)$ for an admissible $F$. Then $\textup{C}(\textup{U}_{\zeta}^+(F) \rtimes \mathbb{T})$ is the universal unital $\textup{C}^*$\nobreakdash-algebra generated by elements $z$ and $u_{ij}$ for $1 \leq i,j \leq n$ subject to
\begin{enumerate}
    \item the relation $zz^*=z^*z=1$,
    \item the commutation relations $zu_{ij}=\zeta^{d_i-d_j}u_{ij}z$, for $1 \leq i,j \leq n$,
    \item and the relations that make the two matrices $u=(u_{ij})_{1 \leq i,j \leq n}$ and $F\overline{u}_{\zeta}F^{-1}$ unitaries, where $\overline{u}_{\zeta}=(\zeta^{d_i(d_j-d_i)}u^*_{ij})_{1 \leq i,j \leq n}$.
\end{enumerate}
Furthermore, the comultiplication $\Delta_{\textup{U}_{\zeta}^+(F) \rtimes \mathbb{T}}$ is given by
\begin{equation}
\Delta_{\textup{U}_{\zeta}^+(F) \rtimes \mathbb{T}}(z)=z \otimes z, \quad \Delta_{\textup{U}_{\zeta}^+(F) \rtimes \mathbb{T}}(u_{ij})=\sum_{k=1}^n u_{ik} \otimes z^{d_k-d_i}u_{kj},
\end{equation} for $1 \leq i,j \leq n$.
\end{theorem}

Before describing the representation theory, let us state a proposition and an important corollary to it.

\begin{proposition}\label{fundamental}
Let $t_{ij}=j_1(z^{d_i})j_2(u_{ij}) \in \textup{C}(\textup{U}_{\zeta}^+(F) \rtimes \mathbb{T})=\textup{C}(\mathbb{T}) \boxtimes_{\zeta} \textup{C}(\textup{U}_{\zeta}^+(F))$ for $1 \leq i,j \leq n$. Then $t=(t_{ij})_{1 \leq i,j \leq n} \in M_n(\textup{C}(\textup{U}_{\zeta}^+(F) \rtimes \mathbb{T}))$ defines a finite dimensional unitary representation of the compact quantum group $\textup{U}_{\zeta}^+(F) \rtimes \mathbb{T}$.   
\end{proposition}

\begin{corollary}\label{cqmg}
The pair $(\textup{C}(\textup{U}_{\zeta}^+(F) \rtimes \mathbb{T}), z \oplus t)$ is a compact matrix quantum group.
\end{corollary}

Finally we discuss the representation theory of the braided free unitary quantum group $\textup{U}^+_{\zeta}(F)$. We begin by observing that the $\textup{C}^*$\nobreakdash-algebra $\textup{C}(\textup{U}^+_{\zeta}(F) \rtimes \mathbb{T})$ is the universal unital $\textup{C}^*$\nobreakdash-algebra generated by $z$ and $t_{ij}$ for $1 \leq i,j \leq n$, subject to the relations that make $z$, $t=(t_{ij})_{1 \leq i,j \leq n}$ and $\overline{t}=(t^*_{ij})_{1 \leq i,j \leq n}$ unitaries, together with the relations $zt_{ij}=\zeta^{d_i-d_j}t_{ij}z$ for $1 \leq i,j \leq n$. Indeed, the proof of Corollary \ref{cqmg} shows that $\overline{u}_{\zeta}$ being equivalent to a unitary is equivalent to $\overline{t}$ being equivalent to a unitary and since $F$ is diagonal, we have the prescribed relations by Theorem \ref{explicitboso}. Furthermore, the relation $zt_{ij}=\zeta^{d_i-d_j}t_{ij}z$ and its adjoint, yield equivalences $z \tenscorep t \cong t \tenscorep z$ and $z \tenscorep \overline{t} \cong \overline{t} \tenscorep z$, respectively.  Now for the matrix $F$, we consider the function algebra $A_u(F)$ of the free unitary quantum group, the generators being denoted by $x_{ij}$, $1 \leq i,j \leq n$. Multiplying the generators $x_{ij}$ by $\zeta^{d_i-d_j}$, $1 \leq i,j \leq n$ yields an automorphism $\alpha$ of the $\textup{C}^*$\nobreakdash-algebra $A_u(F)$. Following again \cite{MR2021}, there is a Hopf $*$\nobreakdash-homomorphism $\phi : A_u(F) \rightarrow \textup{C}(\textup{U}^+_{\zeta}(F) \rtimes \mathbb{T})$ mapping $x_{ij}$ to $t_{ij}$. The proof of the following is similar to Proposition 4.2 of \cite{MR2021}.

\begin{proposition}
The Hopf $*$\nobreakdash-homomorphism $\phi : A_u(F) \rightarrow \textup{C}(\textup{U}^+_{\zeta}(F) \rtimes \mathbb{T})$ extends to an isomorphism $A_u(F) \rtimes_{\alpha} \mathbb{Z} \cong \textup{C}(\textup{U}^+_{\zeta}(F) \rtimes \mathbb{T})$.
\end{proposition}

The Hopf $*$\nobreakdash-homomorphism $\phi$ induces a fully faithful strict tensor functor $\phi_{\ast}$ from the representation category of $A_u(F)$ to that of $\textup{U}^+_{\zeta}(F) \rtimes \mathbb{T}$. The irreducible representations of the free unitary quantum group, as described by Banica \cite{banicaunitary}, are enumerated by $r_{a}$, $a \in \mathbb{N}\ast \mathbb{N}$, with $r_e=1$, $r_{\alpha}=x$, $r_{\beta}=\bar x$, where $\mathbb{N} \ast \mathbb{N}$ is the coproduct in the category of monoids of two copies of $\mathbb{N}$,
generated respectively by $\alpha$ and $\beta$; $e$ denotes the unit of $\mathbb{N} \ast \mathbb{N}$; and $\overline{(\cdot)}$ denotes the involution of $\mathbb{N} \ast \mathbb{N}$ given by $\bar e=e$, $\bar \alpha=\beta$, $\bar \beta=\alpha$ and extended by antimultiplicativity. One furthermore has for any $x$ and $y$ in $\mathbb{N} \ast \mathbb{N}$, $r_{\bar x}=\overline{r_{x}}$ and the fusion rule,
\begin{equation*}
r_{x} \tenscorep r_{y} \cong \bigoplus_{\{a,b,g \in \mathbb{N} \ast \mathbb{N} \mid x=ag, y=\bar g b\}} r_{ab}. 
\end{equation*}
A lemma analogous to Lemma 4.3 of \cite{MR2021} says that the representations $z^x \tenscorep \phi_{\ast}(r_{x'})$ with $x \in \mathbb{Z}$ and $x' \in \mathbb{N} \ast \mathbb{N}$ are all irreducible and distinct; furthermore, any irreducible representation is of this form. And finally, the fusion rules then become
\begin{equation*}
\overline{z^x \tenscorep r_{x'}} \cong \overline{r_{x'}} \tenscorep \overline{z^x} \cong r_{\overline{x'}} \tenscorep z^{-x} \cong z^{-x} \tenscorep r_{\overline{x'}},
\end{equation*} 
\begin{equation*}
(z^x \tenscorep r_{x'}) \tenscorep (z^y \tenscorep r_{y'})\cong z^{x+y} \tenscorep \bigoplus_{\{a,b,g \in \mathbb{N} \ast \mathbb{N} \mid x'=ag, y'=\bar g b\}} r_{ab}. 
\end{equation*}
We note that it is known from \cite{MR2021}, that the representation category of $\textup{U}_{\zeta}^+(F)$ is equivalent to that of the bosonization $\textup{U}_{\zeta}^+(F) \rtimes \mathbb{T}$, thus yielding complete knowledge of the representation category for a diagonal $F$, and in particular, for $\textup{U}_{\zeta}^+(n)$. Collecting these together, we end this article with the following theorem.

\begin{theorem}
Let $F \in \mathrm{GL}(n,\mathbb{C})$ be diagonal and admissible. Then the braided free unitary quantum group $\textup{U}_{\zeta}^+(F)$ has irreducible representations $r_{(x,x')}$, $x \in \mathbb{Z}$, $x' \in \mathbb{N} \ast \mathbb{N}$ such that any irreducible representation is unitarily equivalent to exactly one of these and moreover they satisfy the fusion rule
\begin{equation*}
r_{(x,x')} \tenscorep r_{(y,y')} \cong \bigoplus_{\{a,b,g \in \mathbb{N} \ast \mathbb{N} \mid x'=ag, y'=\bar g b\}} r_{(x+y, ab)}. 
\end{equation*}
\end{theorem}

\section{Symmetries of graph \texorpdfstring{$\textup{C}^*$}{}-algebras}\label{graph} In this section, we come to the main theme of this paper, that of braided symmetries of graph $\textup{C}^*$\nobreakdash-algebras, which rely on the results obtained in the previous sections. We emphasize, however, that almost all the results are modelled on previous works in the unbraided case, as described, for instance, in \cites{JM2021,JM2018} (see also \cite{sw2018}) and we follow the presentations therein. We essentially adapt the techniques developed in the cited references to our braided setting; our contribution is the observation that all the constructs in \cite{JM2018} are $\mathbb{T}$\nobreakdash-equivariant under the generalized gauge action as described below. Having said so, we briefly recall the basic definitions, see \cite{rae2005} for more details. Let $E=(E^0,E^1,r,s)$ be a directed graph. Explicitly, this means that we have the set of vertices $E^0$, the set of edges $E^1$, the range and source maps $r,s : E^1 \rightarrow E^0$, respectively. An edge $e \in E^1$ ``goes from'' its source $s(e) \in E^0$ to its range $r(e) \in E^0$. Usually, the sets $E^1$ and $E^0$ are taken to be countable. A vertex $v \in E^0$ is called a sink (respectively, regular) if $s^{-1}(v)=\{e \in E^1 \mid s(e)=v\}$ is empty (respectively, non-empty and finite). A graph $E$ is row-finite if each of its vertices is either a sink or regular. A path $\alpha$ is a finite sequence $\alpha=e_1\dots e_l$ of edges satisfying $r(e_i)=s(e_{i+1})$ for $i=1,\dots,l-1$. $l$ is called the length of the path, denoted $|\alpha|$, $s(e_1)$ is called the source of the path, denoted $s(\alpha)$ and $r(e_l)$ is called the range of the path, denoted $r(\alpha)$. Now we recall the definition of the graph $\textup{C}^*$\nobreakdash-algebra corresponding to $E$. 

\begin{definition}
Let $E=(E^0,E^1,r,s)$ be a directed graph. The graph $\textup{C}^*$\nobreakdash-algebra $\textup{C}^*(E)$ is the universal $\textup{C}^*$\nobreakdash-algebra generated by families of projections $\{P_v \mid v \in E^0\}$ and partial isometries $\{S_e \mid e \in E^1\}$ subject to the following relations:
\begin{enumerate}
    \item for $v, w \in E^0$, with $v\neq w$, $P_vP_w=0$;
    \item for $e, f \in E^1$, with $e \neq f$, $S^*_eS_f=0$;
    \item for $e \in E^1$, $S^*_eS_e=P_{r(e)}$;
    \item for $e \in E^1$, $S_eS^*_e \leq P_{s(e)}$;
    \item for $v \in E^0$ regular, $P_v=\sum_{e \in s^{-1}(v)}S_eS^*_e$.
\end{enumerate}
\end{definition}

Let $E=(E^0,E^1,r,s)$ be a directed graph and $\textup{C}^*(E)$ be the corresponding graph $\textup{C}^*$\nobreakdash-algebra. $\textup{C}^*(E)$ comes equipped with a natural $\mathbb{T}$\nobreakdash-$\textup{C}^*$\nobreakdash-algebra structure $\rho^{gauge} : \textup{C}^*(E) \rightarrow \textup{C}^*(E) \otimes \textup{C}(\mathbb{T})$, called the gauge action of $\mathbb{T}$ and defined as follows: for $z \in \mathbb{T}$, $v \in E^0$, and $e \in E^1$, $\rho^{gauge}_z(P_v)=P_v$ and $\rho^{gauge}_z(S_e)=zS_e$. For finite graphs $E$, we can slightly generalize the gauge action as follows. 

\begin{definition}\label{gengauge}
Let $\#(E^0)=m$, $\#(E^1)=n$ and $(d_1,\dots,d_n) \in \mathbb{Z}^n$. The generalized gauge action is given by $\rho^{\textup{C}^*(E)}_z(P_{v_i})=P_{v_i}$ and $\rho^{\textup{C}^*(E)}_z(S_{e_j})=z^{d_j}S_{e_j}$ for $z \in \mathbb{T}$, $v_1,\dots,v_m \in E^0$ and $e_1,\dots,e_n \in E^1$. 
\end{definition}

\begin{remark}
For the rest of the section, we only consider finite graphs $E$ without sinks, use the generalized gauge action to equip $\textup{C}^*(E)$ with the structure of a $\mathbb{T}$\nobreakdash-$\textup{C}^*$\nobreakdash-algebra and abuse notation to write $P_i$ instead of $P_{v_i}$, $S_j$ instead of $S_{e_j}$. Furthermore, for a path $\alpha=e_1\dots e_l$, we shall write $S_{\alpha}$ for $S_1\dots S_l$. We will also use the fact repeatedly that $\textup{C}^*(E)$ is the closed linear span of $S_{\alpha}S_{\beta}^*$, where $\alpha, \beta$ are paths in $E$.
\end{remark}

We recall that a $\textup{C}^*$\nobreakdash-dynamical $\mathbb{R}$-system consists of a $\textup{C}^*$\nobreakdash-algebra $X$ and a strongly continuous homomorphism $\sigma : \mathbb{R} \rightarrow \mathrm{Aut}(X)$ of $\mathbb{R}$ in the group of automorphisms of $X$. A self-adjoint operator $x \in \mathcal{L}(\mathcal{H})$ on a finite dimensional Hilbert space $\mathcal{H}$ defines a $\textup{C}^*$\nobreakdash-dynamical $\mathbb{R}$-system $(\mathcal{L}(\mathcal{H}), \mathbb{R}, \sigma)$ which is given by $\sigma_{t}(a)=e^{itx}ae^{-itx}$, $a \in \mathcal{L}(\mathcal{H})$. For such a dynamical system, it is well-known that at any inverse temperature $\beta \in \mathbb{R}$, the unique thermal equilibrium state is given by the Gibbs state 
\begin{equation*}
\omega_{\beta}(a) = \frac{{\mathrm{Tr}}(e^{-\beta x}a)}{{\mathrm{Tr}}(e^{-\beta x})}, \quad a \in X.
\end{equation*} 
For a general $\textup{C}^*$\nobreakdash-dynamical $\mathbb{R}$-system $(X, \mathbb{R},\sigma)$, the generalization of the Gibbs states are the $\mathrm{KMS}$ (Kubo-Martin-Schwinger) states. 

\begin{definition}\cite{br}
A $\mathrm{KMS}$ state for a $\textup{C}^*$\nobreakdash-dynamical $\mathbb{R}$-system $(X,\mathbb{R},\sigma)$ at an inverse temperature $\beta \in \mathbb{R}$ is a state $\tau$ on $X$ that satisfies the $\mathrm{KMS}$ condition given by
\begin{equation}
\tau(ab)=\tau(b\sigma_{i\beta}(a)),
\end{equation}
for $a,b$ in a norm-dense subalgebra $X_{an}$ of $X$ called the algebra of analytic elements of the dynamical system $(X,\mathbb{R},\sigma)$. 
\end{definition}

In this paper, we shall only consider $\mathrm{KMS}$ states on graph $\textup{C}^*$\nobreakdash-algebras with respect to the canonical gauge action. To that end, let $E$ be a finite, directed graph without sinks. We denote the vertex matrix (the $ij$-th entry of which is the number of edges between the $i$-th and $j$-th vertices) by $D$ and the spectral radius of the vertex matrix by $\rho(D)$. This $\rho(D)$ is called the critical inverse temperature \cite{JM2021} and we shall be content with the existence of $\mathrm{KMS}$ states at this critical inverse temperature (see also \cite{kw2013}).

\begin{proposition} \label{exist1}
Let $E=(E^0,E^1,r,s)$ be a finite, directed graph without sinks. Then the graph $\textup{C}^*$\nobreakdash-algebra $\textup{C}^*(E)$ admits a $\mathrm{KMS}_{\log(\rho(D))}$ state, say $\tau_E$ if and only if $\rho(D)$ is an eigenvalue of $D$ such that there is an eigenvector with all entries nonnegative. Furthermore, whenever $\tau_E$ exists, we have the following.
\begin{enumerate}
    \item The vector $(\tau_E(P_{1}),\dots,\tau_E(P_m))$ is an eigenvector corresponding to the eigenvalue $\rho(D)$ of the vertex matrix. Here $P_i$ is the projection corresponding to the vertex $v_i$, $i=1,\dots,m$ (see Definition \textup{\ref{gengauge}}).
    \item The $\mathrm{KMS}$ state $\tau_E$ is given by
    \[
    \tau_E(S_{\alpha}S_{\beta}^*)=\delta_{\alpha\beta}\frac{1}{\rho(D)^{|\alpha|}}\tau_E(P_{r(\alpha)}).
    \]
    Here $|\alpha|$ and $r(\alpha)$ are the length and the range of the path $\alpha$, respectively. The symbol $\delta_{\alpha\beta}$ has the obvious meaning: it vanishes when the two paths $\alpha$ and $\beta$ are different and takes value $1$ when $\alpha$ coincides with $\beta$.
    \item The $\mathrm{KMS}$ state $\tau_E$ is $\mathbb{T}$\nobreakdash-equivariant for the generalized gauge action.
\end{enumerate}
\end{proposition}

\begin{proof}
The demonstrations for the first three conclusions may be found in \cite{JM2021}. For the last, we simply compute. Before that, let us write $d_{\alpha}$ for $d_1+\dots + d_l$, where $\alpha=e_1\dots e_l$ is a path. Then for $z \in \mathbb{T}$, $\rho_z^{\textup{C}^*(E)}(S_{\alpha})=z^{d_{\alpha}}S_{\alpha}$, so that 
\allowdisplaybreaks{
\begin{align*}
\tau_E(\rho_z^{\textup{C}^*(E)}(S_{\alpha}S^*_{\beta}))=&{}\tau_E(z^{d_{\alpha}-d_{\beta}}S_{\alpha}S^*_{\beta})\\
={}&z^{d_{\alpha}-d_{\beta}}\tau_E(S_{\alpha}S^*_{\beta})\\
={}&z^{d_{\alpha}-d_{\beta}}\delta_{\alpha\beta}\frac{1}{\rho(D)^{|\alpha|}}\tau_E(P_{r(\alpha)})\\
={}&\delta_{\alpha\beta}\frac{1}{\rho(D)^{|\alpha|}}\tau_E(P_{r(\alpha)})=\tau_E(S_{\alpha}S^*_{\beta}).
\end{align*}
}The fourth equality can be argued in the following way. If the two paths $\alpha$ and $\beta$ coincide, then $z^{d_{\alpha}-d_{\beta}}$ vanishes. And if $\alpha$ and $\beta$ are different, then the right-hand side of the third equality vanishes.
\end{proof}

\begin{remark}\label{dagger}
The class of graphs satisfying the condition of the Proposition \ref{exist1} is quite big. For example, it contains all the regular graphs. Let us denote the following condition by a $(\dagger)$
\begin{align*}\tag{$\dagger$}
&\rho(D) \text{ is an eigenvalue of } D \text{ such that there is an eigenvector with all }\\
&\text{entries nonnegative} 
\end{align*} 
and refer to the graphs satisfying it as graphs satisfying condition $(\dagger)$. Thus a $\mathrm{KMS}_{\log(\rho(D))}$ state exists if and only if the graph satisfies $(\dagger)$.
\end{remark}

We now recall a few definitions.

\begin{definition}\label{action}\cite{R2021}
Let $G=(\textup{C}(G),\rho^{\textup{C}(G)},\Delta_G)$ be a braided compact quantum group. 
\begin{enumerate}
    \item An action of $G$ on a $\mathbb{T}$\nobreakdash-$\textup{C}^*$\nobreakdash-algebra $(B,\rho^B)$ is a morphism $\eta^B \in \mathrm{Mor}^{\mathbb{T}}(B, B \boxtimes_{\zeta} \textup{C}(G))$ such that
    \begin{enumerate}
    \item $(\mathrm{id}_B \boxtimes_{\zeta} \Delta_G)\circ \eta^B=(\eta^B \boxtimes_{\zeta} \mathrm{id}_{\textup{C}(G)})\circ \eta^B$ (coassociativity);
    \item $\eta^B(B)(1_B \boxtimes_{\zeta} \textup{C}(G))=B \boxtimes_{\zeta} \textup{C}(G)$ (Podle\'s condition).
    \end{enumerate}
    \item Let $(B,\rho^B)$ be a $\mathbb{T}$\nobreakdash-$\textup{C}^*$\nobreakdash-algebra equipped with a $G$\nobreakdash-action $\eta^B \in \mathrm{Mor}^{\mathbb{T}}(B,B\boxtimes_{\zeta} \textup{C}(G))$. A $\mathbb{T}$\nobreakdash-equivariant state $f : B \rightarrow \mathbb{C}$ on $B$ is one that satisfies 
    \[(f \otimes \mathrm{id}_{\textup{C}(\mathbb{T})})\rho^B(b)=f(b)1_{\textup{C}(\mathbb{T})} \text{ for all }b \in B.\]
    Such an $f : B \rightarrow \mathbb{C}$ is said to be preserved under the $G$\nobreakdash-action $\eta^B$ if 
    \[(f \boxtimes_{\zeta} \mathrm{id}_{\textup{C}(G)})\eta^B(b)=f(b)1_{\textup{C}(G)} \text{ for all }b \in B.\]
    \item A $G$\nobreakdash-action $\eta^B \in \mathrm{Mor}^{\mathbb{T}}(B,B\boxtimes_{\zeta} \textup{C}(G))$ of $G$ on a $\mathbb{T}$\nobreakdash-$\textup{C}^*$\nobreakdash-algebra $(B,\rho^B)$ is said to be faithful if the $*$-algebra generated by $\{(f \boxtimes_{\zeta} \mathrm{id}_{\textup{C}(G)})\eta^B(B) \mid f : B \rightarrow \mathbb{C} \text{ a } \mathbb{T}\text{-equivariant state}\}$ is norm-dense in $\textup{C}(G)$.
\end{enumerate}
\end{definition}

\begin{definition}
Let $E=(E^0,E^1,r,s)$ be a finite, directed graph without sinks and let $G=(\textup{C}(G),\rho^{\textup{C}(G)},\Delta_G)$ be a braided compact quantum group. A $G$\nobreakdash-action $\eta \in \mathrm{Mor}^{\mathbb{T}}(\textup{C}^*(E), \textup{C}^*(E) \boxtimes_{\zeta} \textup{C}(G))$ is called linear if there exists $q=(q_{ij})_{1\leq i,j\leq n} \in M_n(\textup{C}(G))$ such that for $1 \leq j \leq n$, we have $\eta(S_j)=\sum_{i=1}^nj_1(S_i)j_2(q_{ij})$.
\end{definition}

We then observe that, by the $\mathbb{T}$\nobreakdash-equivariance of $\eta$ and the homogeneity of $S_j$, $q_{ij}$ is homogeneous of degree $d_j-d_i$ for $1 \leq i,j \leq n$. Furthermore, the coassociativity of $\eta$ yields that $\Delta_G(q_{ij})=\sum_{k=1}^nj_1(q_{ik})j_2(q_{kj})$. We summarize these observations in a lemma.

\begin{lemma}\label{linearstar}
Let $\eta \in \mathrm{Mor}^{\mathbb{T}}(\textup{C}^*(E), \textup{C}^*(E) \boxtimes_{\zeta} \textup{C}(G))$ be a linear action of a braided compact quantum group $G=(\textup{C}(G),\rho^{\textup{C}(G)},\Delta_G)$ on $\textup{C}^*(E)$. Let $q=(q_{ij})_{1\leq i,j\leq n} \in M_n(\textup{C}(G))$ be such that for $1 \leq j \leq n$, $\eta(S_j)=\sum_{i=1}^nj_1(S_i)j_2(q_{ij})$. Then
\begin{enumerate}
    \item for $1 \leq i,j \leq n$ and $z \in \mathbb{T}$, $\rho^{\textup{C}^*(E)}_z(q_{ij})=z^{d_j-d_i}q_{ij}$;
    \item for $1 \leq i,j \leq n$, $\Delta_G(q_{ij})=\sum_{k=1}^nj_1(q_{ik})j_2(q_{kj})$;
    \item for $1 \leq j \leq n$, $\eta(S^*_j)=\sum_{i=1}^n j_1(S^*_i)j_2((\overline{q}_{\zeta})_{ij})$,
\end{enumerate} where $\overline{q}_{\zeta}=(\zeta^{d_i(d_j-d_i)}q^*_{ij})_{1\leq i,j \leq n}$.
\end{lemma}

\begin{proof}
Only the last conclusion needs to be demonstrated. To that end, we observe that for each $1 \leq j \leq n$,
\allowdisplaybreaks{
\begin{align*}
\eta(S^*_j)=\eta(S_j)^*=\sum_{i=1}^nj_2(q^*_{ij})j_1(S_i^*)
={}&\sum_{i=1}^n\zeta^{-d_i(d_i-d_j)}j_1(S_i^*)j_2(q^*_{ij})\\
={}&\sum_{i=1}^n\zeta^{-d_i(d_i-d_j)}j_1(S_i^*)j_2(q^*_{ij})\\
={}&\sum_{i=1}^nj_1(S_i^*)j_2(\zeta^{d_i(d_j-d_i)}q^*_{ij})\\
={}&\sum_{i=1}^nj_1(S_i^*)j_2((\overline{q}_{\zeta})_{ij}),
\end{align*}
}which is what we wanted.
\end{proof}

\begin{definition}
Let $E=(E^0,E^1,r,s)$ be a finite, directed graph without sinks satisfying condition $(\dagger)$ (see Remark \ref{dagger}) and let $\tau_E$ be the $\mathrm{KMS}$ state on $\textup{C}^*(E)$. We define the category $\mathcal{C}(E,\tau_E)$ as follows. 
\begin{enumerate}
\item An object of $\mathcal{C}(E,\tau_E)$ is a pair $(G,\eta)$, where $G=(\textup{C}(G),\rho^{\textup{C}(G)},\Delta_G)$ is a braided compact quantum group, and $\eta \in \mathrm{Mor}^{\mathbb{T}}(\textup{C}^*(E), \textup{C}^*(E) \boxtimes_{\zeta} \textup{C}(G))$ is a $\tau_E$-preserving, linear faithful (see Definition \ref{action}) action of $G$ on $\textup{C}^*(E)$. 
\item Let $(G_1,\eta_1)$ and $(G_2,\eta_2)$ be two objects in $\mathcal{C}(E,\tau_E)$. A morphism $\phi : (G_1,\eta_1) \rightarrow (G_2,\eta_2)$ in $\mathcal{C}(E,\tau_E)$ is by definition a $\mathbb{T}$\nobreakdash-equivariant Hopf $*$-homomorphism $\phi : \textup{C}(G_2) \rightarrow \textup{C}(G_1)$ such that $(\mathrm{id}_{\textup{C}^*(E)} \boxtimes_{\zeta} \phi)\circ \eta_2=\eta_1$.
\end{enumerate} 
\end{definition}

\begin{definition}
A terminal object in $\mathcal{C}(E,\tau_E)$ is called the braided quantum symmetry group of the graph $\textup{C}^*$\nobreakdash-algebra $\textup{C}^*(E)$ and denoted $(\mathrm{Qaut}(\textup{C}^*(E)),\eta^E)$. 
\end{definition}

A priori, it is not clear that $(\mathrm{Qaut}(\textup{C}^*(E)),\eta^E)$ exists but as we shall see below, it indeed does (for the class of  graphs we consider, i.e., finite, directed, without sinks and satisfying the condition $(\dagger)$, so that a $\mathrm{KMS}_{\log(\rho(D))}$ state exists at the critical inverse temperature). This is the main theorem of this section and we shall step by step build up to its proof. Keeping the same notations as above, we recall that by Proposition \ref{exist1}, for $1 \leq i,j \leq n$, 
\[
\tau_E(S_iS^*_j)=\delta_{ij}\frac{1}{\rho(D)}\tau_E(P_{r(e_i)}).
\]Now as observed in \cite{JM2018}, for each $1 \leq i \leq n$, $\tau_E(P_i)\neq 0$, so that we can and in fact do, normalize $S_i$ to obtain $\tau_E(S_iS^*_j)=\delta_{ij}$. Let $\tilde{F}$ be the matrix $(\tau_E(S^*_iS_j))_{1 \leq i,j \leq n}$. Being a KMS state implies that $\tilde{F}$ is an invertible diagonal matrix. 

\begin{proposition}\label{matricial}
Let $\eta \in \mathrm{Mor}^{\mathbb{T}}(\textup{C}^*(E), \textup{C}^*(E) \boxtimes_{\zeta} \textup{C}(G))$ be a linear action of a braided compact quantum group $G=(\textup{C}(G),\rho^{\textup{C}(G)},\Delta_G)$ on $\textup{C}^*(E)$. Let $q=(q_{ij})_{1\leq i,j\leq n} \in M_n(\textup{C}(G))$ be such that for $1 \leq j \leq n$, $\eta(S_j)=\sum_{i=1}^nj_1(S_i)j_2(q_{ij})$. If the $G$\nobreakdash-action $\eta$ preserves $\tau_E$ then for each $1 \leq i,j \leq n$, the following relations hold. 
\begin{enumerate}
    \item $\sum_{k=1}^n\zeta^{d_k(d_j-d_i)}q_{ki}q^*_{kj}=\delta_{ij}$;
    \item $\sum_{k=1}^nq^*_{ki}\tilde{F}_{kk}q_{kj}=\tilde{F}_{ij}$.
\end{enumerate} 
\end{proposition}

\begin{proof}
With the normalization mentioned above, we observe that if the $G$\nobreakdash-action $\eta$ preserves $\tau_E$ then 
\[
(\tau_E \boxtimes_{\zeta} \mathrm{id}_{\textup{C}^*(E)})\eta(S_iS^*_j)=\delta_{ij}1_{\textup{C}(G)}
\]holds for each $1 \leq i,j \leq n$. Therefore, we first write $(\tau_E \boxtimes_{\zeta} \mathrm{id}_{\textup{C}^*(E)})\eta(S_iS^*_j)$ in terms of the $q_{ij}$ $(1 \leq i,j \leq n)$, for which we simply compute:
\allowdisplaybreaks{
\begin{align*}
(\tau_E \boxtimes_{\zeta} \mathrm{id}_{\textup{C}^*(E)})\eta(S_iS^*_j)
=&(\tau_E \boxtimes_{\zeta} \mathrm{id}_{\textup{C}^*(E)})\eta(S_i)\eta(S^*_j)\\
=&(\tau_E \boxtimes_{\zeta} \mathrm{id}_{\textup{C}^*(E)})(\sum_{k,l=1}^nj_1(S_k)j_2(q_{ki})j_1(S^*_l)j_2(\zeta^{d_l(d_j-d_l)}q^*_{lj}))\\
=&(\tau_E \boxtimes_{\zeta} \mathrm{id}_{\textup{C}^*(E)})(\sum_{k,l=1}^nj_1(S_kS^*_l)j_2(\zeta^{d_l(d_j-d_l)-d_l(d_i-d_k)}q_{ki}q^*_{lj}))\\
=&\sum_{k,l=1}^n\tau_E(S_kS^*_l)\zeta^{d_l(d_j-d_l)-d_l(d_i-d_k)}q_{ki}q^*_{lj}\\
=&\sum_{k,l=1}^n\delta_{kl}\zeta^{d_l(d_j-d_l)-d_l(d_i-d_k)}q_{ki}q^*_{lj}\\
=&\sum_{k=1}^n\zeta^{d_k(d_j-d_i)}q_{ki}q^*_{kj},
\end{align*}
}and so we obtain $\sum_{k=1}^n\zeta^{d_k(d_j-d_i)}q_{ki}q^*_{kj}=\delta_{ij}$. For the second conclusion, we observe that if the $G$\nobreakdash-action preserves $\tau_E$ then \[(\tau_E \boxtimes_{\zeta} \mathrm{id}_{\textup{C}^*(E)})\eta(S^*_iS_j)=\tau_E(S^*_iS_j)1_{\textup{C}(G)}\] holds for $1 \leq i,j \leq n$. Now 
\allowdisplaybreaks{
\begin{align*}
(\tau_E \boxtimes_{\zeta} \mathrm{id}_{\textup{C}^*(E)})\eta(S^*_iS_j)
=&(\tau_E \boxtimes_{\zeta} \mathrm{id}_{\textup{C}^*(E)})\eta(S^*_i)\eta(S_j)\\
=&(\tau_E \boxtimes_{\zeta} \mathrm{id}_{\textup{C}^*(E)})(\sum_{k,l=1}^nj_1(S^*_k)j_2(\zeta^{d_k(d_i-d_k)}q^*_{ki}))j_1(S_l)j_2(q_{lj})\\
=&(\tau_E \boxtimes_{\zeta} \mathrm{id}_{\textup{C}^*(E)})(\sum_{k,l=1}^nj_1(S^*_kS_l)j_2(\zeta^{d_k(d_i-d_k)+d_l(d_k-d_i)}q^*_{ki}q_{lj}))\\
=&\sum_{k,l=1}^n\tau_E(S^*_kS_l)\zeta^{d_k(d_i-d_k)+d_l(d_k-d_i)}q^*_{ki}q_{lj}\\
=&\sum_{k,l=1}^n\tilde{F}_{kl}\zeta^{d_k(d_i-d_k)+d_l(d_k-d_i)}q^*_{ki}q_{lj}\\
=&\sum_{k=1}^n\tilde{F}_{kk}q^*_{ki}q_{kj},
\end{align*}   
}and so we obtain $\sum_{k=1}^nq^*_{ki}\tilde{F}_{kk}q_{kj}=\tilde{F}_{ij}$. In the above computation, the last equality holds because $\tilde{F}$ is a diagonal matrix.
\end{proof}

\begin{proposition}\label{quotient}
Let $(G,\eta) \in \mathrm{Obj}(\mathcal{C}(E,\tau_E))$ be an object in the category $\mathcal{C}(E,\tau_E)$. Then the $\textup{C}^*$\nobreakdash-algebra $\textup{C}(G)$ is a quotient of the $\textup{C}^*$\nobreakdash-algebra $\textup{C}(\textup{U}_{\zeta}^{+}(F^{-1}))$, where $F$ is such that $F^*F=\tilde{F}$. 
\end{proposition}

\begin{proof}
Let the linear action $\eta$ of $G$ be given on the generators $S_j$ by $\eta(S_j)=\sum_{i=1}j_1(S_i)j_2(q_{ij})$, for $1 \leq j \leq n$. We now observe that the first relation of Proposition \ref{matricial}, i.e., $\sum_{k=1}^n\zeta^{d_k(d_j-d_i)}q_{ki}q^*_{kj}=\delta_{ij}$ can be written as $(\overline{q}_{\zeta})^*\overline{q}_{\zeta}=I_n$, where $q=(q_{ij})_{1 \leq i,j \leq n}$, $\overline{q}_{\zeta}=(\zeta^{d_i(d_j-d_i)}q^*_{ij})_{1 \leq i,j \leq n}$ and $(\overline{q}_{\zeta})^*$ denotes the usual conjugate i.e., bar-transpose. Indeed,
\allowdisplaybreaks{
\begin{align*}
\sum_{k=1}^n(\overline{q}_{\zeta}^*)_{ik}(\overline{q}_{\zeta})_{kj}={}&\sum_{k=1}^n\zeta^{-d_k(d_i-d_k)}q_{ki}\zeta^{d_k(d_j-d_k)}q^*_{kj}\\
={}&\sum_{k=1}^n\zeta^{-d_k(d_i-d_k)+d_k(d_j-d_k)}q_{ki}q^*_{kj}=\sum_{k=1}^n\zeta^{d_k(d_j-d_i)}q_{ki}q^*_{kj},
\end{align*}
}yielding the required relation. The second relation obtained in Proposition \ref{matricial} $\sum_{k=1}^nq^*_{ki}\tilde{F}_{kk}q_{kj}=\tilde{F}_{ij}$ too can be put in a more compact form by writing $q^*\tilde{F}q=\tilde{F}$. Indeed,
\allowdisplaybreaks{
\begin{align*}
(q^*\tilde{F}q)_{ij}=\sum_{k,l=1}^n(q^*)_{ik}\tilde{F}_{kl}q_{lj}=\sum_{k,l=1}^nq^*_{ki}\delta_{kl}\tilde{F}_{kk}q_{kj}=\sum_{k=1}^nq^*_{ik}\tilde{F}_{kk}q_{kj},    
\end{align*}   
}which is what we wanted. Next, we observe that each of the diagonal entries of $\tilde{F}$ is positive; thus there is indeed an $F$ such that $F^*F=\tilde{F}$. Now we consider $FqF^{-1}$ and compute
\allowdisplaybreaks{
\begin{align*}
(FqF^{-1})^*(FqF^{-1})=(F^*)^{-1}q^*F^*FqF^{-1}=(F^*)^{-1}q^*\tilde{F}qF^{-1},
\end{align*}   
}which says that the relation $q^*\tilde{F}q=F$ is equivalent to $(FqF^{-1})^*(FqF^{-1})=I_n$. Letting $q'=FqF^{-1}$, we now find out what $F^{-1}\overline{q'}_{\zeta}F$ is.
\allowdisplaybreaks{
\begin{align*}
(F^{-1}\overline{q'}_{\zeta}F)_{ij}=\sum_{k,l=1}^nF^{-1}_{ik}(\overline{q'}_{\zeta})_{kl}F_{lj}={}&F^{-1}_{ii}(\overline{q'}_{\zeta})_{ij}F_{jj}\\
={}&F^{-1}_{ii}(\zeta^{d_i(d_j-d_i)}(FqF^{-1})^*_{ij})F_{jj}\\
={}&F^{-1}_{ii}(\zeta^{d_i(d_j-d_i)}F_{ii}q^*_{ij}F^{-1}_{jj})F_{jj}\\
={}&\zeta^{d_i(d_j-d_i)}q^*_{ij},   
\end{align*}   
}i.e., we have $F^{-1}(\overline{q'}_{\zeta})F=\overline{q}_{\zeta}$ and therefore, the relation $(\overline{q}_{\zeta})^*\overline{q}_{\zeta}=I_n$ is equivalent to \[(F^{-1}(\overline{q'}_{\zeta})F)^*F^{-1}(\overline{q'}_{\zeta})F=I_n.\]  Taken together, we see that the matrix $q$ satisfies the relations $(F^{-1}(\overline{q'}_{\zeta})F)^*F^{-1}(\overline{q'}_{\zeta})F=I_n$ and $q'^*q'=I_n$. These two relations still hold after passing to the bosonization $G \rtimes \mathbb{T}$ of $G$. Since $G \rtimes \mathbb{T}$ is a compact quantum group, we obtain the remaining two relations i.e., $F^{-1}(\overline{q'}_{\zeta})F(F^{-1}(\overline{q'}_{\zeta})F)^*=I_n$ and $q'q'^*=I_n$. Taking the four sets of relations together, we have that $q'$ and $F^{-1}(\overline{q'}_{\zeta})F$ are unitaries and the universal property of the $\textup{C}^*$\nobreakdash-algebra $\textup{C}(\textup{U}_{\zeta}^{+}(F^{-1}))$ yields the conclusion.
\end{proof}

\begin{definition}\label{anothercat}
We now define another category $\mathcal{C}'(E)$ as follows. An object of $\mathcal{C}'(E)$ is a triple $(X,\rho^X,\eta^X)$ which consists of 
\begin{enumerate}
    \item a $\mathbb{T}$\nobreakdash-$\textup{C}^*$\nobreakdash-algebra $(X,\rho^X)$ generated by $\{t_{ij}\}_{1 \leq i,j \leq n}$ such that the two matrices $FtF^{-1}=(F_{ii}t_{ij}F^{-1}_{jj})_{1 \leq i,j \leq n}$ and $\overline{t}_{\zeta}=(\zeta^{d_i(d_j-d_i)}t^*_{ij})_{1 \leq i,j \leq n}$ are unitaries.
    \item a $\mathbb{T}$\nobreakdash-equivariant morphism $\eta^X \in \mathrm{Mor}^{\mathbb{T}}(\textup{C}^*(E), \textup{C}^*(E) \boxtimes_{\zeta} X)$ such that, for each $1 \leq j \leq n$, $\eta^X(S_j)=\sum_{i=1}^nj_1(S_i)j_2(t_{ij})$.
\end{enumerate}
Let $(X,\rho^X,\eta^X)$ and $(Y,\rho^Y,\eta^Y)$ be two objects in the category $\mathcal{C}'(E)$. A morphism $\phi : (X,\rho^X,\eta^X) \rightarrow (Y,\rho^Y,\eta^Y)$ in $\mathcal{C}'(E)$ is by definition a $\mathbb{T}$\nobreakdash-equivariant morphism $\phi : X \rightarrow Y$ such that $(\mathrm{id}_{\textup{C}^*(E)} \boxtimes_{\zeta} \phi)\circ \eta^X=\eta^Y$. 
\end{definition}

\begin{remark}
We remark that by the $\mathbb{T}$\nobreakdash-equivariance of $\eta$ and the homogeneity of $S_j$, $t_{ij}$ is homogeneous of degree $d_j-d_i$ for $1 \leq i,j \leq n$.
\end{remark}

\begin{lemma}
An initial object in the category $\mathcal{C}'(E)$ exists.
\end{lemma}

\begin{proof}
We begin by remarking again that the proof is essentially the same as given in \cite{JM2018}. We recall that for the matrix $F$, $\textup{C}(\textup{U}_{\zeta}^{+}(F^{-1})$ is the universal $\textup{C}^*$\nobreakdash-algebra with generators $u_{ij}$ $(1 \leq i,j \leq n)$ subject to the relations that make $u$ and $F^{-1} \overline{u}_{\zeta}F$ unitaries, where $u=(u_{ij})_{1 \leq i,j \leq n}$ and $\overline{u}_{\zeta}=(\zeta^{d_i(d_j-d_i)}u^*_{ij})_{1 \leq i,j \leq n}$. Thus for any object $(X,\rho^X,\eta^X) \in \mathrm{Obj}(\mathcal{C}'(E))$, by the universal property of $\textup{C}(\textup{U}_{\zeta}^{+}(F^{-1}))$, there is a surjective $*$-homomorphism $\pi_X : \textup{C}(\textup{U}_{\zeta}^{+}(F^{-1})) \rightarrow X$ sending $u_{ij}$ to $F_{ii}t_{ij}F^{-1}_{jj}$. As remarked above, for $1 \leq i,j \leq n$, $t_{ij}$ is homogeneous of degree $d_j-d_i$ and so $\pi_X$ is $\mathbb{T}$\nobreakdash-equivariant. We define $\mathcal{I}=\bigcap_{\Lambda}{\ker}(\pi_X)$, where $\Lambda=\{X : (X,\rho^X,\eta^X) \in \mathrm{Obj}(\mathcal{C}'(E))\}$ and observe that $\mathcal{I}$ is a $\mathbb{T}$\nobreakdash-invariant closed two-sided ideal of $\textup{C}(\textup{U}_{\zeta}^{+}(F^{-1}))$ so that $\textup{C}(\textup{U}_{\zeta}^{+}(F^{-1}))/\mathcal{I}$ makes sense. We denote $\textup{C}(\textup{U}_{\zeta}^{+}(F^{-1}))/\mathcal{I}$ by $\mathcal{U}$ and let $\pi : \textup{C}(\textup{U}_{\zeta}^{+}(F^{-1})) \rightarrow \mathcal{U}$ be the quotient map. The $\mathbb{T}$\nobreakdash-invariance of $\mathcal{I}$ induces a $\mathbb{T}$\nobreakdash-action on $\mathcal{U}$ which can also be described explicitly as follows. We write $[a]=\pi(a)$ for the class of $a \in \textup{C}(\textup{U}_{\zeta}^{+}(F^{-1}))$ in $\mathcal{U}$. Then for $z \in \mathbb{T}$, $\rho^{\mathcal{U}}_z([u_{ij}])=[z^{d_j-d_i}u_{ij}]$. Thus $(\mathcal{U},\rho^{\mathcal{U}})$ is indeed a $\mathbb{T}$\nobreakdash-$\textup{C}^*$\nobreakdash-algebra. Once we show that there is a $\mathbb{T}$\nobreakdash-equivariant morphism $\eta^{\mathcal{U}}$ satisfying the requirements as in Definition \ref{anothercat}, the conclusion that $(\mathcal{U},\rho^{\mathcal{U}},\eta^{\mathcal{U}})$ is an initial object in $\mathcal{C}'(E)$ is immediate. To that end, let us set $\eta^{\mathcal{U}}(S_j)=\sum_{i=1}^nj_1(S_i)j_2([F^{-1}_{ii}u_{ij}F_{jj}])$. It is a tedious but straightforward check that $\eta^{\mathcal{U}}$ is indeed a morphism from $\textup{C}^*(E)$ to $\textup{C}^*(E) \boxtimes_{\zeta} X$ and it is clearly $\mathbb{T}$\nobreakdash-equivariant, thus implying that $\eta^{\mathcal{U}} \in \mathrm{Mor}^{\mathbb{T}}(\textup{C}^*(E), \textup{C}^*(E) \boxtimes_{\zeta} X)$. This finishes the proof.
\end{proof}

Let us keep the notations from the above proof; thus $(\mathcal{U},\rho^{\mathcal{U}},\eta^{\mathcal{U}})$ is the initial object of $\mathcal{C}'(E)$. We shall provide $\mathcal{U}$ with more structures so as to make it a braided compact quantum group and prove that it is the braided quantum symmetry group $(\mathrm{Qaut}(\textup{C}^*(E)),\eta^E)$. We first observe that $\mathcal{U} \boxtimes_{\zeta} \mathcal{U}$ can be made into an object of the category $\mathcal{C}'(E)$. 

\begin{lemma}
The braided tensor product $\mathcal{U} \boxtimes_{\zeta} \mathcal{U}$ of $\mathcal{U}$ with itself can be made into an object $(\mathcal{U} \boxtimes_{\zeta} \mathcal{U},\rho^{\mathcal{U} \boxtimes_{\zeta} \mathcal{U}}, \eta^{\mathcal{U} \boxtimes_{\zeta} \mathcal{U}})$ of $\mathcal{C}'(E)$.
\end{lemma}

\begin{proof}
Indeed, letting $t_{ij}=\sum_{k=1}^nj_1([u_{ik}])j_2([u_{kj}])$ for $1 \leq i,j \leq n$, we see that the $\textup{C}^*$\nobreakdash-algebra $\mathcal{U} \boxtimes_{\zeta} \mathcal{U}$ is generated by $t_{ij}$ $(1 \leq i,j \leq n)$. A repetition of the proof of Proposition \ref{coproduct} then yields that $t_{ij}$ satisfy the required relations as in Definition \ref{anothercat}. Finally, the morphism $\eta^{\mathcal{U} \boxtimes_{\zeta} \mathcal{U}}=(\eta^{\mathcal{U}} \boxtimes_{\zeta} \mathrm{id}_{\mathcal{U}})\circ \eta^{\mathcal{U}} \in \mathrm{Mor}^{\mathbb{T}}(\textup{C}^*(E), \textup{C}^*(E) \boxtimes_{\zeta} \mathcal{U} \boxtimes_{\zeta} \mathcal{U})$ satisfies 
\allowdisplaybreaks{
\begin{align*}
\eta^{\mathcal{U} \boxtimes_{\zeta} \mathcal{U}}(S_j)={}&\sum_{k,i=1}^nj_1(S_i)j_2(F^{-1}_{ii}[u_{ik}]F_{kk})j_3(F^{-1}_{kk}[u_{kj}]F_{jj})\\
={}&\sum_{i=1}^nj_1(S_i)j_2(F^{-1}_{ii}t_{ij}F_{jj}),
\end{align*}   
}which implies that $(\mathcal{U} \boxtimes_{\zeta} \mathcal{U},\rho^{\mathcal{U} \boxtimes_{\zeta} \mathcal{U}}, \eta^{\mathcal{U} \boxtimes_{\zeta} \mathcal{U}})$ is an object of $\mathcal{C}'(E)$.
\end{proof}

\begin{corollary}
There exists a unique $\mathbb{T}$\nobreakdash-equivariant morphism $\Delta_{\mathcal{U}} \in \mathrm{Mor}^{\mathbb{T}}(\mathcal{U}, \mathcal{U} \boxtimes_{\zeta} \mathcal{U})$ such that $\Delta_{\mathcal{U}}([u_{ij}])=\sum_{k=1}^nj_1([u_{ik}])j_2([u_{kj}])$. Furthermore, $\Delta_{\mathcal{U}}$ is coassociative and bisimplifiable \textup{(}see Definition \textup{\ref{bcqg})}.
\end{corollary}

\begin{proof}
Since $\mathcal{U}$ is the initial object in the category $\mathcal{C}'(E)$, there is a unique morphism $\Delta_{\mathcal{U}}$ from $(\mathcal{U},\rho^{\mathcal{U}},\eta^{\mathcal{U}})$ to $(\mathcal{U} \boxtimes_{\zeta} \mathcal{U},\rho^{\mathcal{U} \boxtimes_{\zeta} \mathcal{U}}, \eta^{\mathcal{U} \boxtimes_{\zeta} \mathcal{U}})$. Explicitly, this means that $\Delta_{\mathcal{U}}$ is a $\mathbb{T}$\nobreakdash-equivariant morphism $\Delta_{\mathcal{U}} : \mathcal{U} \rightarrow \mathcal{U} \boxtimes_{\zeta} \mathcal{U}$ such that $(\mathrm{id}_{\textup{C}^*(E)} \boxtimes_{\zeta} \Delta_{\mathcal{U}})\circ \eta^{\mathcal{U}}=\eta^{\mathcal{U} \boxtimes_{\zeta} \mathcal{U}}$. The last equality when evaluated at $S_j$ $(1 \leq j \leq n)$, together with their linear independence, yields $\Delta_{\mathcal{U}}([u_{ij}])=\sum_{k=1}^nj_1([u_{ik}])j_2([u_{kj}])$. The arguments for coassociativity and bisimplifiability are exactly similar as in the proof of the Proposition \ref{coproduct}.
\end{proof}

\begin{corollary}
There exists a braided compact quantum group \textup{(}over $\mathbb{T}$\textup{)}, to be denoted $G_E$, such that $(\textup{C}(G_E),\rho^{\textup{C}(G_E)},\Delta_{G_E})=(\mathcal{U},\rho^{\mathcal{U}},\Delta_{\mathcal{U}})$. Furthermore, $G_E$ acts linearly, faithfully on $\textup{C}^*(E)$ preserving $\tau_E$ via $\eta^{\mathcal{U}}$, denoted henceforth by $\eta^E$.
\end{corollary}

\begin{proof}
The first statement is essentially renaming. For the second, we observe that $\eta^{\mathcal{U}}$ is coassociative. Indeed, by definition $\eta^{\mathcal{U} \boxtimes_{\zeta} \mathcal{U}}=(\eta^{\mathcal{U}} \boxtimes_{\zeta} \mathrm{id}_{\mathcal{U}})\circ \eta^{\mathcal{U}}$ and we have $(\mathrm{id}_{\textup{C}^*(E)} \boxtimes_{\zeta} \Delta_{\mathcal{U}})\circ \eta^{\mathcal{U}}=\eta^{\mathcal{U} \boxtimes_{\zeta} \mathcal{U}}=(\eta^{\mathcal{U}} \boxtimes_{\zeta} \mathrm{id}_{\mathcal{U}})\circ \eta^{\mathcal{U}}$. The Podle\'s condition can be checked again along the same lines as in the proof of Proposition \ref{coproduct}. Finally, $\eta^{\mathcal{U}}$ is, once again by definition, linear, faithful and it preserves $\tau_E$ by Proposition \ref{matricial}. This completes the proof.
\end{proof}

We end this section with the following main existence theorem.

\begin{theorem}\label{graphsymm}
Let $E=(E^0,E^1,r,s)$ be a finite, directed graph without sinks that satisfies the condition $(\dagger)$. Then $(\mathrm{Qaut}(\textup{C}^*(E)),\eta^E)$ exists.
\end{theorem}

\begin{proof}
As in \cite{JM2018}, we shall show that $G_E$ is the terminal object in the category $\mathcal{C}(E,\tau_E)$ and thus is isomorphic to $(\mathrm{Qaut}(\textup{C}^*(E)),\eta^E)$. To that end, let $(G,\eta) \in \mathrm{Obj}(\mathcal{C}(E,\tau_E))$ be an object in the category $\mathcal{C}(E,\tau_E)$. By Proposition \ref{quotient}, the triple $(\textup{C}(G),\rho^{\textup{C}(G)},\eta^{\textup{C}(G)}=\eta)$ is an object of the category $\mathcal{C}'(E)$, hence there is a unique $\phi \in \mathrm{Mor}^{\mathbb{T}}(\textup{C}(G_E), \textup{C}(G))$ such that $(\mathrm{id}_{\textup{C}^*(E)} \boxtimes_{\zeta} \phi)\circ \eta^{\textup{C}(G_E)}=\eta^{\textup{C}(G)}$. This is equivalent to saying that $G_E$ is indeed the terminal object in the category $\mathcal{C}(E,\tau_E)$.
\end{proof}

\section{Symmetries of the Cuntz algebra}\label{sec:cuntz-algebra} Having proved the existence of the braided quantum symmetry group, we now explicitly compute it for the Cuntz algebra $\mathcal{O}_n$. We recall that the Cuntz algebra $\mathcal{O}_n$ is the graph $\textup{C}^*$\nobreakdash-algebra corresponding to the graph (denoted by $E_{\mathcal{O}_n}$) with a single vertex and $n$-loops at it. Explicitly, $\mathcal{O}_n$ is the universal $\textup{C}^*$\nobreakdash-algebra generated by $S_i$ for $1 \leq i \leq n$ subject to the relations
\begin{equation}\label{cuntz}
S^*_iS_j=\delta_{ij} \quad (1 \leq i,j \leq n), \text{ and } S_1S_1^*+\dots+S_nS_n^*=1.
\end{equation} $\mathcal{O}_n$ is equipped with the generalized gauge action $\rho^{\mathcal{O}_n} : \mathcal{O}_n \rightarrow \mathcal{O}_n \otimes \textup{C}(\mathbb{T})$ given by $\rho^{\mathcal{O}_n}_z(S_i)=z^{d_i}S_i$, $1 \leq i \leq n$, $z \in \mathbb{T}$, and $\underline{d}=(d_1,\dots,d_n) \in \mathbb{Z}^n$. The next Proposition shows that the braided free unitary quantum group $\textup{U}_{\zeta}^{+}(n)$ acts on $\mathcal{O}_n$.

\begin{proposition}\label{unitaryaction}
There is a unique unital $*$-homomorphism 
\[
\eta^{\mathcal{O}_n} : \mathcal{O}_n \rightarrow \mathcal{O}_n \boxtimes_{\zeta} \textup{C}(\textup{U}_{\zeta}^+(n))
\]such that $\eta^{\mathcal{O}_n}(S_j)=\sum_{i=1}^nj_1(S_i)j_2(u_{ij})$ for $1 \leq i,j \leq n$. Furthermore, $\eta^{\mathcal{O}_n}$ is $\mathbb{T}$\nobreakdash-equivariant, coassociative and satisfies Podle\'s condition \textup{(}see Definition \textup{\ref{action})}. 
\end{proposition}

\begin{proof}
Let $S_j'=\sum_{i=1}^nj_1(S_i)j_2(u_{ij})$ for $1 \leq i,j \leq n$. We remark that each $S_j'$ is homogeneous of degree $d_j$. We first observe that for each $1 \leq j \leq n$
\allowdisplaybreaks{
\begin{align*}
S_j'^*={}&\sum_{i=1}^nj_2(u^*_{ij})j_1(S_i^*)\\
={}&\sum_{i=1}^n\zeta^{-d_i(d_i-d_j)}j_1(S_i^*)j_2(u^*_{ij})\\
={}&\sum_{i=1}^nj_1(S_i^*)j_2(\zeta^{d_i(d_j-d_i)}u^*_{ij})=\sum_{i=1}^nj_1(S_i^*)j_2((\overline{u}_{\zeta})_{ij}).
\end{align*}
}Now, by the universal property, we see that a (necessarily unique) $*$-homomorphism $\eta^{\mathcal{O}_n} : \mathcal{O}_n \rightarrow \mathcal{O}_n \boxtimes_{\zeta} \textup{C}(\textup{U}_{\zeta}^+(n))$ satisfying $\eta^{\mathcal{O}_n}(S_j)=S_j'$ exists if and only if $S_j'$ (for $1 \leq j \leq n$) satisfy Eq.\eqref{cuntz}. To see that $S_j'$ for $1 \leq j \leq n$ indeed satisfy Eq.\eqref{cuntz}, we compute
\allowdisplaybreaks{
\begin{align*}
S_i'^*S_j'={}&\sum_{\alpha, \beta=1}^nj_1(S_{\alpha}^*)j_2(\zeta^{d_{\alpha}(d_i-d_{\alpha})}u^*_{\alpha i})j_1(S_{\beta})j_2(u_{\beta j})\\
={}&\sum_{\alpha, \beta=1}^nj_1(S_{\alpha}^*)j_1(S_{\beta})\zeta^{d_{\beta}(d_{\alpha}-d_i)}\zeta^{d_{\alpha}(d_i-d_{\alpha})}j_2(u^*_{\alpha i})j_2(u_{\beta j})\\
={}&\sum_{\alpha, \beta=1}^nj_1(S_{\alpha}^*S_{\beta})\zeta^{(d_{\beta}-d_{\alpha})(d_{\alpha}-d_i)}j_2(u^*_{\alpha i}u_{\beta j})\\
={}&\sum_{\alpha, \beta=1}^n\delta_{\alpha \beta}\zeta^{(d_{\beta}-d_{\alpha})(d_{\alpha}-d_i)}j_2(u^*_{\alpha i}u_{\beta j})\\
={}&\sum_{\alpha=1}^nj_2(u^*_{\alpha i}u_{\alpha j})=\delta_{ij},
\end{align*}
}The fourth and fifth equalities use the relations of $\mathcal{O}_n$ and that $u$ is a unitary, respectively. Similarly,
\allowdisplaybreaks{
\begin{align*}
\sum_{j=1}^nS_j'S_j'^*={}&\sum_{j,\alpha,\beta=1}^nj_1(S_{\alpha})j_2(u_{\alpha j})j_1(S^*_{\beta})j_2(\zeta^{d_{\beta}(d_j-d_{\beta})}u^*_{\beta j})\\
={}&\sum_{j,\alpha,\beta=1}^nj_1(S_{\alpha})j_1(S^*_{\beta})\zeta^{-d_{\beta}(d_j-d_{\alpha})}\zeta^{d_{\beta}(d_j-d_{\beta})}j_2(u_{\alpha j})j_2(u^*_{\beta j})\\
={}&\sum_{\alpha,\beta=1}^nj_1(S_{\alpha}S^*_{\beta})\zeta^{d_{\beta}d_{\alpha}-d_{\beta}^2}j_2(\sum_{j=1}^nu_{\alpha j}u^*_{\beta j})\\
={}&\sum_{\alpha,\beta=1}^nj_1(S_{\alpha}S^*_{\beta})\zeta^{d_{\beta}d_{\alpha}-d_{\beta}^2}\delta_{\alpha\beta}\\
={}&\sum_{\alpha=1}^nj_1(S_{\alpha}S^*_{\alpha})=1
\end{align*}
}Here also, the fourth and fifth equalities use that $u$ is a unitary and the relations of $\mathcal{O}_n$, respectively. Therefore, we have constructed a unique and unital $*$-homomorphism $\eta^{\mathcal{O}_n} : \mathcal{O}_n \rightarrow \mathcal{O}_n \boxtimes_{\zeta} \textup{C}(\textup{U}_{\zeta}^+(n))$ satisfying $\eta^{\mathcal{O}_n}(S_j)=S_j'$ for $1 \leq j \leq n$. As remarked above, for $1 \leq j \leq n$, $S_j'$ is homogeneous of degree $d_j$ and so $\eta^{\mathcal{O}_n}$ is $\mathbb{T}$\nobreakdash-equivariant. The coassociativity and the Podle\'s can be proved along the same lines as in the proof of Proposition \ref{coproduct}.
\end{proof}

Now for the Cuntz algebra $\mathcal{O}_n$, the vertex matrix $D$ is just a scalar $n$ and so $\rho(D)=n$. And therefore, the underlying graph satisfies the condition $(\dagger)$ of Remark \ref{dagger}, implying the existence of the $\mathrm{KMS}_{\log n}$ state $\tau_{E_{\mathcal{O}_n}}$, which we denote by $\tau_n$, to shorten the notation. By Proposition \ref{exist1}, $\tau_n(S_iS^*_j)=\delta_{ij} \frac{1}{n}$. In fact, on the dense $*$\nobreakdash-algebra spanned by $S_{\alpha}S^*_{\beta}$, where $\alpha,\beta$ are paths, $\tau_n$ is given by $\tau_n(S_{\alpha}S^*_{\beta})=\delta_{\alpha \beta}\frac{1}{n^{|\alpha|}}$. We recall the following lemma from \cite{JM2021}.

\begin{lemma}\label{cuntzkms}\cite{JM2021}
The $\mathrm{KMS}_{\log n}$ state $\tau_n$ satisfies $\tau_n(S_{\alpha}xS^*_{\beta})=\delta_{\alpha \beta}\frac{1}{n^{|\alpha|}}\tau_n(x)$ for all $x \in \mathcal{O}_n$ and paths $\alpha,\beta$ with $|\alpha|=|\beta|$.
\end{lemma}

\begin{proposition}\label{objectproof}
The pair $(\textup{U}_{\zeta}^+(n), \eta^{\mathcal{O}_n})$ is an object of the category $\mathcal{C}(E_{\mathcal{O}_n}, \tau_n)$.    
\end{proposition}

\begin{proof}
By Proposition \ref{unitaryaction}, we already have that the $\textup{U}_{\zeta}^+(n)$-action is linear and faithful. To finish the proof, we need to show that $\eta^{\mathcal{O}_n}$ preserves $\tau_n$. To that end, we first remark that it suffices to check the equivariance on $S_{\alpha}S^*_{\beta}$, where $\alpha,\beta$ are paths, i.e., 
\begin{equation}\label{toshow}
(\tau_n \boxtimes_{\zeta} \mathrm{id}_{\mathcal{O}_n})\eta^{\mathcal{O}_n}(S_{\alpha}S^*_{\beta})=\tau_n(S_{\alpha}S^*_{\beta})1_{\textup{C}(\textup{U}_{\zeta}^+(n))}.
\end{equation}
Now there are two cases.

\subsection*{Case 1.} $|\alpha|\neq |\beta|$. In this case, the right-hand side of Eq.\eqref{toshow} is $0$. For the left-hand side, we observe that $\eta^{\mathcal{O}_n}(S_{\alpha}S^*_{\beta})=\eta^{\mathcal{O}_n}(S_{\alpha})\eta^{\mathcal{O}_n}(S^*_{\beta})$. Using linearity of $\eta^{\mathcal{O}_n}$ and the commutation relations of the imbeddings $j_1$ and $j_2$, we find a $u_{\alpha \beta \mu \nu} \in \textup{C}(\textup{U}_{\zeta}^+(n))$, such that $\eta^{\mathcal{O}_n}(S_{\alpha}S^*_{\beta})$ is a finite sum of elements of the form $j_1(S_{\mu}S^*_{\nu})j_2(u_{\alpha \beta \mu \nu})$, where $\mu$ and $\nu$ are multi-indices of the form $(i_1,\dots,i_{|\alpha|})$ and $(j_1,\dots,j_{|\beta|})$, respectively. Then $(\tau_n \boxtimes_{\zeta} \mathrm{id}_{\mathcal{O}_n})\eta^{\mathcal{O}_n}(S_{\alpha}S^*_{\beta})$ is a finite sum of elements of the form $\tau_n(S_{\mu}S^*_{\nu})u_{\alpha \beta \mu \nu}$. Since $|\alpha|\neq |\beta|$, $\tau_n(S_{\mu}S^*_{\nu})$ vanishes too, settling this case for good.

\subsection*{Case 2.} $|\alpha|=|\beta|$. For this case, we use induction on $|\alpha|=|\beta|$. Thus let $\textbf{P}(k)$ be the statement ``Eq.\eqref{toshow} holds for paths $\alpha$ and $\beta$ of lengths $|\alpha|=|\beta|=k$''.

\subsection*{Step 1.} We prove that $\textbf{P}(1)$ holds. For this, we observe that length $1$ paths are the generators themselves, i.e., $S_{\alpha}=S_i$ and $S_{\beta}=S_j$ for some $1 \leq i,j \leq n$. Since $u$ and $\overline{u}_{\zeta}$ are unitaries, a repetition of the proof of Proposition \ref{matricial} settles this step.

\subsection*{Step 2.} Now we assume that $\textbf{P}(m)$ holds for all $m < m_0$, i.e., Eq.\eqref{toshow} holds for paths $\alpha$ and $\beta$ of lengths $|\alpha|=|\beta|=m < m_0$. We have to prove that $\textbf{P}(m_0)$ holds. 

\subsection*{Step 2a.} We remark that $\tau_n(S_{\alpha}S^*_{\beta})=0$ if $S_{\alpha}S^*_{\beta}$ has nonzero homogeneous degree, i.e., in the notation of the proof of Proposition \ref{exist1}, $d_1+\dots + d_{|\alpha|}=d_{\alpha} \neq d_{\beta}=d_1+\dots + d_{|\beta|}$. Indeed, $d_{\alpha}\neq d_{\beta}$ implies $\alpha \neq \beta$ and so $\tau_n(S_{\alpha}S^*_{\beta})=\delta_{\alpha\beta}\frac{1}{n^{|\alpha|}}$ vanishes.

\subsection*{Step 2b.} With $\textbf{Step 2a}$ in hand, we take paths $\alpha$ and $\beta$ with $|\alpha|=|\beta|=m_0$. We write $S_{\alpha}S^*_{\beta}$ as $S_ixS^*_j$ with $x$ of the form $S_{\alpha^{\prime}}S^*_{\beta^{\prime}}$ where $\alpha^{\prime}$ and $\beta^{\prime}$ are paths of lengths $m_0-1$. We note that Lemma \ref{cuntzkms} can be applied and so $\tau_n(S_{\alpha}S^*_{\beta})=\tau_n(S_ixS^*_j)=\delta_{ij}\frac{1}{n}\tau_n(x)$. Now the left-hand side of Eq.\eqref{toshow}, when evaluated at $S_ixS^*_j$, reduces to the following.
\allowdisplaybreaks{
\begin{align*}
&(\tau_n \boxtimes_{\zeta} \mathrm{id}_{\mathcal{O}_n})\eta^{\mathcal{O}_n}(S_ixS^*_j)\\
={}&(\tau_n \boxtimes_{\zeta} \mathrm{id}_{\mathcal{O}_n})\eta^{\mathcal{O}_n}(S_i)\eta^{\mathcal{O}_n}(x)\eta^{\mathcal{O}_n}(S^*_j)\\
={}&(\tau_n \boxtimes_{\zeta} \mathrm{id}_{\mathcal{O}_n})(\sum_{(x)}\sum_{k,l=1}^nj_1(S_k)j_2(u_{ki})j_1(x_0)j_2(x_1)j_1(S^*_l)j_2(\zeta^{d_l(d_j-d_l)}u^*_{lj}))\\
={}&(\tau_n \boxtimes_{\zeta} \mathrm{id}_{\mathcal{O}_n})(\sum_{(x)}\sum_{k,l=1}^nj_1(S_kx_0S^*_l)j_2(u_{ki}x_1u^*_{lj}))\\
&\zeta^{d_l(d_j-d_l)+\mathrm{deg}(x_0)(d_i-d_k)-d_l(d_i-d_k+\mathrm{deg}(x_1))}\\
={}&\sum_{(x)}\sum_{k,l=1}^n\tau_n(S_kx_0S^*_l)u_{ki}x_1u^*_{lj}\zeta^{d_l(d_j-d_l)+\mathrm{deg}(x_0)(d_i-d_k)-d_l(d_i-d_k+\mathrm{deg}(x_1))}\\
={}&\sum_{(x)}\sum_{k,l=1}^n\delta_{kl}\frac{1}{n}\tau_n(x_0)u_{ki}x_1u^*_{lj}\zeta^{d_l(d_j-d_l)+\mathrm{deg}(x_0)(d_i-d_k)-d_l(d_i-d_k+\mathrm{deg}(x_1))}\\
={}&\sum_{(x)}\sum_{k=1}^n\frac{1}{n}\tau_n(x_0)u_{ki}x_1u^*_{kj}\zeta^{d_k(d_j-d_k)+\mathrm{deg}(x_0)(d_i-d_k)-d_k(d_i-d_k+\mathrm{deg}(x_1))}\tag{$\ddagger$}
\end{align*}   
}The second equality uses the definition of $\eta^{\mathcal{O}_n}$ (see Proposition \ref{unitaryaction}), Sweedler notation for $\eta^{\mathcal{O}_n}$ and Lemma \ref{linearstar}; the third equality uses commutation relations for the imbeddings $j_1$ and $j_2$; the fifth equality uses Lemma \ref{cuntzkms}. We recall that we have assumed $\textbf{P}(m)$ holds for $m < m_0$; thus $(\tau_n \boxtimes_{\zeta} \mathrm{id}_{\mathcal{O}_n})\eta^{\mathcal{O}_n}(x)=\tau_n(x)1_{\textup{C}(\textup{U}_{\zeta}^+(n))}$, which, in Sweedler notation as used above, is equivalent to $\sum_{(x)}\tau_n(x_0)x_1=\tau_n(x)1_{\textup{C}(\textup{U}_{\zeta}^+(n))}$. We now further divide into two more steps.

\subsection*{Step 2b\texorpdfstring{${}^\prime$}{}} We first assume that $x$ has nonzero homogeneous degree. Then $(\ddagger)$ becomes
\allowdisplaybreaks{
\begin{align*}
&\sum_{(x)}\sum_{k=1}^n\frac{1}{n}\tau_n(x_0)u_{ki}x_1u^*_{kj}\zeta^{d_k(d_j-d_k)+\mathrm{deg}(x_0)(d_i-d_k)-d_k(d_i-d_k+\mathrm{deg}(x_1))}\\
={}&\sum_{(x)}\sum_{k=1}^n\frac{1}{n}u_{ki}\tau_n(x)u^*_{kj}\zeta^{d_k(d_j-d_k)+\mathrm{deg}(x_0)(d_i-d_k)-d_k(d_i-d_k+\mathrm{deg}(x_1))}\\
={}&0,
\end{align*}    
}by $\textbf{Step 2a}$ above. Thus the left-hand side of Eq.\eqref{toshow} is $0$. The right-hand side is \[\tau_n(S_ixS^*_j)1_{\textup{C}(\textup{U}_{\zeta}^+(n))}=\delta_{ij}\frac{1}{n}\tau_n(x)1_{\textup{C}(\textup{U}_{\zeta}^+(n))}=0\] again by $\textbf{Step 2a}$, confirming this case.

\subsection*{Step 2b\texorpdfstring{${}^{\prime\prime}$}{}.} We now assume that $\mathrm{deg}(x)=0$. Now in $(\ddagger)$, we observe that under the sum $\sum_{(x)}$, only those $x_0$ survive for which $\mathrm{deg}(x_0)=0$, by $\textbf{Step 2a}$. But $\mathrm{deg}(x)=\mathrm{deg}(x_0)=0$ imply $\mathrm{deg}(x_1)=0$. Then $(\ddagger)$ reduces to the following.
\allowdisplaybreaks{
\begin{align*}
&\sum_{(x)}\sum_{k=1}^n\frac{1}{n}\tau_n(x_0)u_{ki}x_1u^*_{kj}\zeta^{d_k(d_j-d_k)+\mathrm{deg}(x_0)(d_i-d_k)-d_k(d_i-d_k+\mathrm{deg}(x_1))}\\
={}&\sum_{(x)}\sum_{k=1}^n\frac{1}{n}\tau_n(x_0)u_{ki}x_1u^*_{kj}\zeta^{d_k(d_j-d_k)-d_k(d_i-d_k)}\\
={}&\sum_{k=1}^n\frac{1}{n}\tau_n(x)u_{ki}u^*_{kj}\zeta^{d_k(d_j-d_k)-d_k(d_i-d_k)}\\
={}&\frac{1}{n}\tau_n(x)\sum_{k=1}^nu_{ki}u^*_{kj}\zeta^{d_k(d_j-d_i)}\\
={}&\frac{1}{n}\tau_n(x)\delta_{ij}.
\end{align*}   
}The fourth equality uses the fact that $\overline{u}_{\zeta}$ is a unitary. Therefore, the right-hand side of Eq.\eqref{toshow} equals $\frac{1}{n}\tau_n(x)\delta_{ij}$ which, by Lemma \ref{cuntzkms} is the left-hand side of Eq.\eqref{toshow}. This completes the step, hence the induction, and the proof. 
\end{proof}

\begin{theorem}\label{cuntzsymm}
The pair $(\textup{U}_{\zeta}^+(n), \eta^{\mathcal{O}_n})$ is the terminal object in the category $\mathcal{C}(E_{\mathcal{O}_n}, \tau_n)$, i.e., $(\textup{U}_{\zeta}^+(n), \eta^{\mathcal{O}_n}) \cong (\mathrm{Qaut}(\mathcal{O}_n),\eta^{\mathcal{O}_n})$.
\end{theorem}

\begin{proof}
The relations Eq.\eqref{cuntz} imply that the matrix $\tilde{F}=\tau_n(S^*_iS_j)_{1 \leq i,j \leq n}$ defined in the proof of Proposition \ref{matricial} is just the identity matrix $I_n$. The result now follows from Propositions \ref{matricial} and \ref{unitaryaction}.
\end{proof}

\begin{remark}
The bosonization $\textup{U}_{\zeta}^+(n) \rtimes \mathbb{T}$ acts on the $\textup{C}^*$\nobreakdash-algebra $\textup{C}(\mathbb{T}) \boxtimes_{\zeta} \mathcal{O}_n$, which is isomorphic to the crossed product $\mathcal{O}_n \rtimes \mathbb{Z}$ of $\mathcal{O}_n$ with $\mathbb{Z}$, the $\mathbb{Z}$\nobreakdash-action being given by the automorphism induced by $\zeta^{-1}$. In fact, we expect it to be the universal object in a suitable category but as the main focus of this article is the braided quantum group itself, as opposed to its bosonization, we refrain from going further in this direction.
\end{remark}

\end{document}